\documentclass[twoside]{article}

\usepackage{amsmath}
\usepackage{amscd}
\usepackage{graphics}
\usepackage{epsfig}
\usepackage{url}

\usepackage{pgf}
\usepackage{tikz}
\usetikzlibrary{arrows,shapes,er,automata,backgrounds,topaths,trees}

\usepackage{mor}

\newcommand*{\CopyCounter}[2]{%
  \expandafter\def\csname c@#2\endcsname{\csname c@#1\endcsname}%
  \expandafter\def\csname p@#2\endcsname{\csname p@#1\endcsname}%
  \expandafter\def\csname the#2\endcsname{\csname the#1\endcsname}}

\newcounter{Observation}
\numberwithin{Observation}{section}
\newtheorem{observ}[Observation]{Observation}

\newcommand{\bm}[1]{\mbox{\boldmath${#1}$}}
\newcommand{\Assume}[1]{{\bf (A#1)}}

\newcommand{\domain}{\Omega}
\newcommand{\cdomain}{\widebar{\Omega}}
\newcommand{\boundary}{\partial \domain}

\newcommand{\xbar}{\bm{\bar{x}}}

\newcommand{\xtilde}{\bm{\tilde{x}}}

\newcommand{\bx}[1]{\mbox{\boldmath$x_{#1}$}}

\newcommand{\x}{\bm{x}}
\newcommand{\T}{\bm{t}}

\newcommand{\ba}{\bm{a}}

\newcommand{\y}{\bm{y}}
\newcommand{\z}{\bm{z}}
\newcommand{\e}{\bm{e}}

\newcommand{\widebar}{\overline}
\newcommand{\R}{\bm{R}} 

\newcommand{\J}{{\cal{J}}} 
\newcommand{\Cost}{\J}

\newcommand{\M}{{\cal{M}}} 
\newcommand{\XX}{{X \backslash \{ \T \}}}

\DeclareMathOperator*{\argmin}{argmin}

\received{}                              %
\revised{}                               %
\pubyear{0x}                             %
\pubmonth{Xxxxxxx}                       %
\volume{xx}                              %
\issue{x}                                %
\pages{xxx--xxx}                         %
\firstpage{0xxx}                         %
\DOI{10.1287/moor.xxxx.xxxx}             %
\startpagenumber{1}                      %

\title{Label-setting methods for Multimode Stochastic Shortest Path problems on graphs.}
\ShortTitle{Label-setting methods for MSSP}
\ShortAuthors{Vladimirsky}
\NumberOfAuthors{1}
\FirstAuthor{Alexander Vladimirsky}
\FirstAuthorAddress{Department of Mathematics, Cornell University, Ithaca, NY 14853}
\FirstAuthorEmail{vlad@math.cornell.edu}
\FirstAuthorURL{http://www.math.cornell.edu/~vlad/}


\MSCcodes{
90C39,
90C40,
49L20, 
65N22,
35B37;
93E20,
65K05,
49L25 
} 

\ORMScodes{Primary: dynamic programming; 
Secondary: Markov finite state} 

\begin{document}
\maketitle

\begin{abstract}
Stochastic shortest path (SSP) problems arise in a variety
of discrete stochastic control contexts.
An optimal solution to such a problem is typically
computed using the value function, which can be found by solving
the corresponding dynamic programming equations.
In the deterministic case, these equations can 
be often solved by the highly efficient label-setting methods
(such as Dijkstra's and Dial's algorithms).
In this paper
we define and study a class of Multimode Stochastic Shortest Path
problems and develop sufficient conditions for the applicability
of label-setting methods.  We illustrate our approach on a number
of discrete stochastic control examples.  We also discuss
the relationship of SSPs with discretizations of static 
Hamilton-Jacobi equations and provide an alternative derivation
for several fast (non-iterative) numerical methods for these PDEs.
\end{abstract}
\normalsize

\section{Introduction.}
\label{s:intro}

Stochastic Shortest Path (SSP) problems constitute a large class 
of Markov Decision Processes and their accurate and efficient
solution is important for numerous applications including 
mathematical finance, 
optimal resource allocation,
design of discrete-time risk-sensitive controls, and 
controlled queuing in communication networks. 
Our goal in this paper is to study the conditions under which 
an important subclass of SSPs
can be solved by efficient label-setting methods.

In SSP the current state of the system at the 
$k-$th stage is $\y_k$, an  element in a finite 
state space $X= \{\x_1, ... , \x_M, \T = \x_{M+1}  \}$. 
At the next stage, $\y_{k+1}$ is
a random variable, whose probability distribution on $X$
depends on $\y_k$ and on the decision made (control value chosen) 
at the previous stage.  The process terminates upon reaching a special
target state $\T$.  At each stage, our choice of control determines
the incurred cost and the overall goal is to minimize the ``value function''
(i.e., the expected value of the total accumulated cost up to the termination).
We provide a formal description of SSP in section \ref{s:SSP_general}; 
here we simply note
that the dynamic programming approach yields a system of $M$ coupled
nonlinear equations for the value function.  Under mild technical assumptions
this system has a solution, which can be found by ``value iteration''.
However, since these iterations are performed in $\R^M$, this can be quite
costly, especially considering the fact that infinitely many steps
are generally needed for convergence.

On the other hand, SSPs can be considered as a generalization 
of classical deterministic shortest path (SP) problems on directed graphs,
for which there is a variety of well-understood efficient algorithms.
In particular, non-iterative {\em label-setting} methods are applicable
provided the transition-costs in the graph are non-negative.
If a constant $\kappa << M$ is an upper bound on outdegrees,
Dijkstra's method \cite{Diks} and Dial's method \cite{Dial}
solve the deterministic dynamic programming equations
in $O(M \log M)$ and $O(M)$ operations respectively.
We provide a brief overview of these methods in section \ref{ss:LS_determ}, 
but here we simply note that both methods hinge on the {\em absolute causality}
present in a deterministic problem:  the fact that the value function is 
decreasing along every optimal path to $\T$.

Thus, to build similar methods for SSPs, one needs to find similar
causal properties in the stochastic problem.  
In fact, Bertsekas showed that a Dijkstra-like method will
correctly compute the value function of an SSP if 
there exists a {\em consistently improving optimal policy}
\cite[Vol. II, p.98]{Bertsekas_DPbook}.
In Section \ref{ss:LS_general_SSP} we define a similar
notion of a {\em consistently $\delta$-improving optimal policy},
which guarantees the applicability of a Dial-like method.
Unfortunately, both of these criteria are implicit since 
the existence of such optimal policies is generally not
known a priori. 

The main contribution of this paper is the development of explicit 
conditions on transition cost function(s), which guarantee the existence 
of consistently improving and/or consistently $\delta$-improving 
optimal policies for a large class of {\em Multimode SSPs}.
The exact class of SSPs that we consider is 
formally defined in Section \ref{s:MSSP},
but generally our criteria apply provided
\begin{enumerate}
\item
each state $\x \in X$ has a collection of ``modes'' $m_1(\x), \ldots, m_r(\x)$ -- 
(possibly overlapping) subsets of $X$;
\item
each individual control is restricted to one of the modes 
(i.e., has non-zero transition probabilities only 
into the states available in that mode);
\item
there exists an available control corresponding to
each possible probability distribution over the states in each mode;
\item
the control-cost is defined for each mode separately
as a continuous function of the corresponding probability distribution
over the states in that mode.
\end{enumerate}
In this setting, it is natural to interpret the decision made at 
each stage as a deterministic choice among the modes of $\y_k$
plus the choice of a desirable probability distribution for 
the transition to one of the states in that mode.

This class obviously includes the classical SP problem
(when each mode contains only one possible successor-state).
More interestingly, it includes the problem of selecting optimal
randomized/mixed controls for deterministic shortest path problems,
when such randomized/mixed controls might be available at a discount.
In section \ref{ss:MSSP_modeling} we consider several 
representative examples
and discuss the differences between {\em explicitly causal} problems
(where the causality stems from a particularly simple structure
of transition probabilities) and {\em absolutely causal} problems 
(where the applicability of label-setting methods stems from certain
properties of transition costs, as derived in section \ref{ss:cost_criteria}).

The Multimode SSPs also include (but are not limited to)
the Markov chain approximations of deterministic continuous optimal 
trajectory problems.  (E.g., consider a vehicle starting
at some point $\x$ inside the domain $\domain \subset \R^n$, 
which is controlled to minimize the time needed to reach the boundary
$\boundary$.)  The value function for such problems is 
typically found as a viscosity solution of a static 
first-order Hamilton-Jacobi PDE.
It is well-known that semi-Lagrangian discretizations of that PDE
(similar to those in \cite{Falcone_InfHor} and \cite{GonzalezRofman}) 
can also be obtained from controlled Markov processes on 
the underlying grids.  This approach was pioneered by Kushner and Dupuis
\cite{KushnerDupuis} to design approximation schemes for deterministic
and stochastic continuously-controlled processes.  Recent extensions 
include higher-order approximations  \cite{SzpiroDupuis} and methods 
for stochastic differential games \cite{Kushner_games}.  
The resulting systems of equations are typically treated iteratively,
but relatively recently provably convergent label-setting 
algorithms were introduced for several important subclasses.  
For the isotropic 
case (when the vehicle's speed depends only on its current position in 
$\domain$ and is independent of the chosen direction of motion), 
the corresponding PDE is Eikonal.  In 1994 Tsitsiklis introduced
the first Dijkstra-like and Dial-like methods for semi-Lagrangian 
discretizations of this PDE on a uniform Cartesian grid 
\cite{Tsitsiklis_conference, Tsitsiklis}.  
The family of Dijkstra-like Fast Marching Methods, 
introduced by Sethian in \cite{SethFastMarcLeveSet}
and extended by Sethian and co-authors in
\cite{SethSIAM, KimmSethTria, SethVlad1}, was developed for Eulerian upwind 
discretizations of the Eikonal PDE in the context of 
isotropic front propagation problems.
A detailed discussion of similarities and differences of
these approaches can be found in \cite{SethVlad3}.
More recently, another Dial-like method for the Eikonal PDE on a uniform grid
was introduced in \cite{KimGMM}.  
For the anisotropic case,
the resulting semi-Lagrangian discretization typically 
does not have that causal property and the label-setting
methods are not directly applicable.
The label-setting Ordered Upwind Methods \cite{SethVlad2, SethVlad3} 
circumvent this difficulty; the key idea behind them can be interpreted 
as ``modifying the computational stencil on-the-fly to ensure the causality''.  
In the appendix of \cite{SethVlad3} we also demonstrated that the
causality is present for the first-order semi-Lagrangian discretizations
of the Eikonal PDE on arbitrary acute meshes.

In all of the above cases the proofs of causality heavily relied
on a geometric interpretation of the problem (e.g., a discretization of 
a particular PDE on a specific grid or mesh in $\R^n$).
In contrast, we first demonstrate that the applicability of
label-setting methods to Multimode SSPs
can be proven even if no geometric interpretation is available
(Section \ref{s:MSSP}).
We then show that the absolute causality of prior numerical methods
for the Eikonal PDE 
can be easily re-derived from the more general criteria introduced in here.
In addition, our formulation yields two new results
for deterministic continuous optimal trajectory problems  (Section \ref{s:HJB_discr}):\\
$\bullet \,$ a formula for the bucket-width in a Dial-like method for Eikonal PDEs on acute meshes;\\
$\bullet \,$ an applicability criterion for the label-setting techniques in anisotropic 
optimal control problems.\\
Finally, in Section \ref{s:conclusions} we discuss the limitations
of our approach and list several related open problems.

\section{Stochastic Shortest Path Problem.}
\label{s:SSP_general}

Typically SSP is described on a directed graph  
with nodes $X = \{\x_1, ... , \x_M, \T = \x_{M+1}  \}.$  Our exposition here 
closely follows the standard setting 
described in \cite{Bertsekas_DPbook}.

For each $\bx{i}$ the problem specifies a compact set of allowable 
controls $A_i = A(\x_i)$.  
If $\bx{i}$ is the current state of the process (i.e., if $\y_k = \x_i$), 
then our choice of a control value $\ba \in A_i$ determines the cost of the next transition 
$C(\bx{i}, \ba)$ as well as the probability of transition into each node $\bx{j}$:
$$
p(\x_i, \x_j, \ba) = p_{ij}(\ba) = 
P \left( \y_{k+1} = \x_j \, \mid \, 
\y_k = \x_i, \text{ and the chosen control is } \ba \in A_i \right).
$$  
A class of problems where the transition cost $C(\x_i, \ba, \x_j)$ 
also depends on the resulting successor node $\x_j$ can also be
handled in the same framework by defining
$C(\x_i, \ba) = \sum\limits_{j = 1}^{M+1} p_{ij}(\ba) C(\x_i, \ba, \x_j).$
It is assumed that the cost is accumulated until we reach the  
{\em absorbing target} $\T$, i.e., 
$p_{tt}(\ba) = 1$ and $C(\T, \ba) = 0$ for $\forall \ba \in A_t$.

Consider the class of control mappings 
$\mu: X \mapsto \left(\bigcup\limits_{i=1}^M A_i \right)$ such that
$\mu(x_i) \in A_i$ for all $\x_i \in X$.  A {\em policy} is an infinite sequence of
such mappings $\pi = (\mu_0, \mu_1, \ldots).$  A {\em stationary policy} is a policy of 
the form $(\mu, \mu, \ldots)$ and for the sake of brevity we will also refer to it 
as ``the stationary policy $\mu$''.

If the process starts at $\x \in X$ (i.e., $\y_0 = \x$), the expected cost of using 
a policy $\pi = (\mu_0, \mu_1, \ldots)$ is defined as
$$\Cost(\x, \pi) = E \left( \sum\limits_{k = 0}^{\infty} C(\y_k, \mu_k(\y_k)) \right).$$
The value function is then defined as usual 
$U_i = U(\x_i) = \inf\limits_{\pi} \Cost(\x_i, \pi)$, and a 
policy $\pi^*$ is called {\em optimal} provided 
$U(\x_i) = \Cost(\x_i, \pi^*)$ for all $\x_i \in X$.

If the value function $U(\x_i)$ is 
finite, it should satisfy the
dynamic programming equations:  $U(\T) = 0$ and
\begin{equation}
\label{eq:Markov_general}
U_i \; = \; \inf\limits_{\ba \in A_i} 
\left\{ 
C(\x_i, \ba) \, + \, 
\sum\limits_{j = 1}^{M+1}  p_{ij}(\ba) U_j 
\right\}, \qquad \text{for } \forall \x_i \in \XX.
\end{equation}

\noindent
An operator $T$ is defined on $\R^M$ component-wise by applying 
the right hand side of equation (\ref{eq:Markov_general}); i.e., for
any $W \in \R^M$
\begin{equation}
(TW)_i \; = \; \inf\limits_{\ba \in A_i} 
\left\{ 
C(\x_i, \ba) \, + \, 
\sum\limits_{j = 1}^{M+1}  p_{ij}(\ba) W_j 
\right\}.
\label{eq:generic_value_it_component}
\end{equation}
Clearly, 
$U = 
\left[
\begin{array}{c}
U_1\\
\vdots\\
U_M
\end{array}
\right]
$
is a fixed point of $T$ and one hopes to recover $U$ 
by {\em value iteration}:

\begin{equation}
W^{n+1} \, := \, T \, W^n \qquad
\text{ starting from an initial guess $W^0 \in \R^M$.}
\label{eq:generic_value_it}
\end{equation}
However, $T$ generally is not a contraction unless 
all stationary policies are known to be proper \cite{BerTsi_NDP}.

In \cite{BertsTsi} Bertsekas and Tsitsiklis demonstrated 
the existence of a stationary optimal policy, the uniqueness
of the fixed point of $T$, and that $W^n \rightarrow U$
for arbitrary $W^0 \in \R^M$ under the following 
assumptions:

\noindent 
$\bullet \,$
\Assume{0} All $C(\bx{i}, \ba)$ are lower-semicontinuous and all 
$p_{ij}(\ba)$ are continuous functions of controls $\ba$.

\noindent 
$\bullet \,$
\Assume{1} There exists at least one {\em proper policy} (i.e., a policy, which 
reaches the target $\T$ with probability 1 regardless of the initial state 
$\x \in X$).

\noindent 
$\bullet \,$
\Assume{2} Every improper policy $\pi$ will have cost $\Cost(\x, \pi) = + \infty$ for 
at least one node $\x \in X$.

\vspace*{1mm}
\noindent
\Assume{0} and the compactness of control sets $A_i$
allow us to replace ``$\inf$'' with ``$\min$'' in formulas
(\ref{eq:Markov_general}) and (\ref{eq:generic_value_it_component}).
\Assume{1} corresponds to a graph connectivity assumption in the deterministic case
while \Assume{2} is a stochastic analog of requiring all cycles to have positive cumulative
penalty.  \Assume{2} also follows automatically if 
$$
\underline{C} \; = \;
\min\limits_{\x \in \XX, \, \ba \in A(\x)} \;  C(\x, \ba) \quad > \quad 0.
$$ 

The convergence of value iteration provides a way for computing $U$,
but generally that convergence does not occur after any finite number of iterations
(for a simple example, see Figure \ref{fig:_infinite_VI}).
Some error bounds are available, but typically in an implicit form only 
\cite[Vol. I, Section 7.2]{Bertsekas_DPbook}.
A recent work by Bonet \cite{Bonet} provides a polynomial upper bound on the number
of value iterations required to achieve a prescribed accuracy for the case 
when the ratio $(\|U\|_{\infty} \, / \, \underline{C})$ 
is a priori known to be polynomially bounded.

\begin{figure}[hhhh]
\center{
\begin{tikzpicture}[->,>=stealth',shorten >=1pt,auto,node distance=3cm,
                    semithick]
  \tikzstyle{every state}=[draw, shape=circle]

  \node[state]         (A)							{$\x_1$};
  \node[state]         (C) [right of=A]		{$\T$};
  \node[state]         (B) [right of=C]		{$\x_2$};
  
  \path (A) edge	[bend left=45]		node {$p_{12}=1/2$} (B)
            edge						node {$p_{1t}=1/2$} (C)
        (B) edge	[bend left=45]		node {$p_{21}=1/2$} (A)
            edge	[swap]				node {$p_{2t}=1/2$} (C);
\end{tikzpicture}
}
\caption{
{\footnotesize
A simple example with only one available control at every node;
the transition probabilities are indicated above and the control cost is 
$C>0$.
By the symmetry it is clear that $U_1 = U_2 = u$ and $u = C + \frac{1}{2}(u + 0)$;
thus, $u = 2C$.  At the same time, the generic value iteration described by formula
(\ref{eq:generic_value_it}) will not converge after any finite number of steps
unless $W^0 = U$.
}
}
\label{fig:_infinite_VI}
\end{figure}
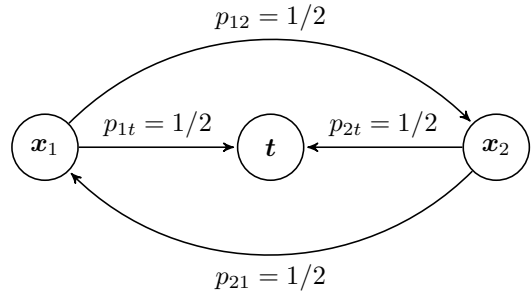

\subsection{Label-setting methods: the deterministic case.}
\label{ss:LS_determ}
Fast methods for deterministic discrete control problems 
(e.g., searching for a shortest path in a graph or a network) 
can be found in all standard references
(e.g., \cite{Ahuja}, \cite{Bertsekas_NObook}) and we provide 
a brief overview just for the sake of completeness.
The dynamic programming equations
are much simpler in this case: $U(\T) = 0$ and
\begin{equation}
\label{eq:Optimality_Determ}
U_i \; = \; \min\limits_{\bx{j} \in N(\bx{i})} \left\{C_{ij} + U_j\right\},
 \qquad \text{for } \forall \x_i \in \XX
\end{equation}
where $U(\bx{i}) = U_i$ is the min-cost-to-exit starting from $\bx{i}$,
$N(\bx{i})$ is the set of nodes to which $\bx{i}$ is connected, and
$C_{ij} = C(\bx{i}, \bx{j})$ is the cost of traversing the corresponding link.
In the absence of negative cost cycles and if every $\x_i$ is connected by some path
to $\T$, the value function is finite and well-defined.
Value iteration (\ref{eq:generic_value_it}) will converge to $U$ after at most $M$ iterations
resulting in a $O(M^2)$ computational cost.

Label-setting methods provide a better alternative if a suitable lower bound on 
the transition costs is available.  These methods reorder the iterations over 
the nodes to guarantee that each $U_i$ is recomputed at most 
$\kappa$ times, where 
the constant upper bound on outdegrees $\kappa$ 
is assumed to be much smaller than the total number of nodes $M$.
For example, Dijkstra's classical algorithm \cite{Diks} is a label-setting method
for the deterministic case provided  all $C_{ij} \geq 0$.
The idea is based on the causality of the system (\ref{eq:Optimality_Determ}): 
\begin{equation}
\label{cond:causality_determ}
U_i \text{ may depend on $U_j$ only if } \; U_i \geq U_j.
\end{equation}  
Such an ordering is not known in advance and 
has to be obtained at run-time.  
The method subdivides $X$ into two classes: ``permanently labeled nodes'' $P$
and ``tentatively labeled nodes'' $L$ and the values for $\x_i$'s in $L$ are 
successively re-evaluated using only the adjacent values already in $P$: 
\begin{equation}
\label{eq:temp_labels_Determ}
U(\x_i) \; := \; \min\limits_{\bx{j} \in \tilde{N}(\bx{i})} \left\{C_{ij} + U_j\right\},
 \qquad \qquad \text{for }  \x_i \in L,
\end{equation}
where $\tilde{N}(\bx{i}) = N(\x_i) \cap P$.
The algorithm is initialized by placing all nodes into $L$ and 
setting $U(\T) = 0$ and $U(\x_i) =  + \infty$
for $i = 1, \ldots, M$.
At each stage the algorithm chooses the smallest of tentative labels $U(\xbar)$,
``accepts'' $\xbar$ (i.e., moves $\xbar$ from $L$ to $P$), and
re-evaluates $U_i$ for each $\x_i \in L$ such that $\xbar \in N(\x_i).$
Since $\xbar$ is the only new element in $\tilde{N}(\bx{i})$, that re-evaluation
can be more efficiently performed by setting
\begin{equation}
\label{eq:temp_labels_local_Determ}
U(\x_i) \; := \; \min \left\{ U(\x_i), \left( C(\x_i, \xbar) + U(\xbar) \right) \right\}.
\end{equation}
The algorithm terminates once the list $L$ is empty, at which
point the vector $U \in \R^M$ satisfies the system of equations 
(\ref{eq:Optimality_Determ}). 
The necessity to sort all (finite) temporary labels  
dictates the use of heap-sort data structures \cite{Ahuja}, usually resulting
in the overall computational complexity of $O(M \log M)$.  

In addition, if all $C_{ij} \geq \Delta > 0$, then Dial's label-setting method is also 
applicable \cite{Dial}. The idea is to avoid sorting temporarily labeled nodes and instead 
place them into ``buckets'' of
width $\Delta$ based on their tentative values.  If $U(\xbar)$ is the ``smallest'' of
tentative labels and $U(\x)$ is currently in the same bucket, then
even after $\xbar$ is permanently labeled, it cannot affect $U(\x)$ since
$$U(\x)  < U(\xbar) + \Delta \leq U(\xbar) + C(\x, \xbar).$$
Thus, a typical stage of Dial's method consists of 
``accepting'' (or declaring labels to be permanent) for
everything in the current bucket, recomputing all nodes in $L$ adjacent to 
those newly labeled permanent, switching them to other buckets if warranted by the 
new tentative labels, and then moving on to the next bucket.   
Since inserting to and deleting from a bucket
can be performed in $O(1)$ time, the overall computational complexity of Dial's method
becomes $O(M)$.  In addition, while Dijkstra's approach is inherently sequential,
Dial's method is naturally parallelizable.  Many other enhancements of the above 
label-setting methods are available in the literature (e.g., see 
\cite{Bertsekas_NObook} and references therein).  Most of those enhancements can be
also used with the label-setting of SSP -- provided the basic versions of the above 
algorithms are applicable.

\subsection{Label-setting methods: the general SSP.}
\label{ss:LS_general_SSP}
Given a stationary policy $\mu$ for a general SSP, 
we can construct its directed {\em dependency graph}
$G_{\mu}$ using the nodes $\XX$ and 
connecting $\x_i$ to $\x_j$ if $p_{ij}(\mu(x_i)) > 0$.
Assuming \Assume{0}, \Assume{1} and \Assume{2},
it is easy to show that the value iteration for this problem
converges after at most $M$ iterations provided there exists 
an optimal stationary policy $\mu^*$ such that 
$G_{\mu^*}$ is acyclic.  
(See \cite[Vol. II, Section 2.2.1]{Bertsekas_DPbook}).
We will refer to such SSPs as {\em causal}.

\begin{remark}
\label{rem:self_transition}
This condition seems to forbid any self-transitions
(e.g., $p_{ii} (\mu^*(x_i)) > 0$ for $\forall \x_i \neq \T$), but an SSP
with self-transitions can be converted into an SSP without them by setting
$$
\tilde{C}(\x_i, \ba) \, = \, \frac{C(\x_i, \ba)}{1 - p_{ii}(\ba)};
\qquad \text{and} \qquad
\tilde{p}_{ij}(\ba) \, = \,
\begin{cases}
0 & \text{if } i = j,\\
\frac{p_{ij}(\ba)}{1 - p_{ii}(\ba)}
& \text{if } i \neq j;
\end{cases}
\qquad
\text{for } 
\begin{array}{l}
i=1,\cdots,M;\\
j=1,\cdots,M+1.
\end{array}
$$
\end{remark}

\begin{remark}
\label{rem:explicit_causality}
One obvious set of causal SSPs consists of all problems where
the dependency graph is acyclic for {\em every} stationary policy.
The SSP belongs to this class if and only if, for $\forall \x_i,\x_j \in \XX$,
$$
\exists \mu_1 \text{ s.t. there is a path from } \x_i \text{ to } \x_j \text{ in } G_{\mu_1}
\qquad
\Longrightarrow
\qquad
\not \exists \mu_2 \text{ s.t. there is a path from } \x_j \text{ to } \x_i \text{ in } G_{\mu_2}.
$$
We will refer to such problems as {\em explicitly causal}
(see examples \ref{ex:aux1} and \ref{ex:multitask} in the later sections).
Such explicit causality is independent of the cost function
and can be determined based on the available controls and the transition 
probabilities alone.  The above property imposes a partial 
ordering on $X$; using that partial ordering to go through the nodes,
we will clearly have the value function computed correctly
in a single sweep, yielding the computational complexity of $O(M)$.
Thus, the applicability of label-setting methods described below
is only important for SSPs which are causal, but not explicitly causal.
This is similar to the fact that the original Dijkstra's method is 
not needed to solve the deterministic SP problem on any acyclic digraph.
\end{remark}

\noindent
According to the definition introduced by Bertsekas in \cite{Bertsekas_DPbook},
an optimal stationary policy $\mu^*$ is {\em consistently improving} if
\begin{equation}
\label{def:improving}
p_{ij}(\mu^*(\x_i)) > 0 \qquad \Longrightarrow \qquad U_i > U_j.
\end{equation}
This is a stochastic equivalent of the causality condition 
(\ref{cond:causality_determ}).
Thus, the existence of such $\mu^*$ not only guarantees that 
$G_{\mu^*}$ is acyclic, but also allows us
to avoid the value iteration process altogether 
since $U_i$'s can be computed by a non-iterative Dijkstra-like method instead. 

If a consistently improving optimal policy is known to exist,
the new equivalent of the causal update equation (\ref{eq:temp_labels_Determ}) 
for each $\x_i \in L$ is now
\begin{equation}
\label{eq:temp_labels_SSP}
U(\x_i) \; := \; \min\limits_{\ba \in \tilde{A}(\x_i)} 
\left\{ 
C(\x_i, \ba) \, + \, 
\sum\limits_{j = 1}^{M+1}  p_{ij}(\ba) U_j 
\right\},
\end{equation}
where $\tilde{A}(\x_i)$ is the set of controls, which make transition
possible to permanently labeled nodes only, i.e.,
$\tilde{A}(\x_i) = 
\{ \ba \in A(\x_i) \mid p_{ij}(\ba) = 0 
\text{ for all } \x_j \not \in P \}$.
Once $\xbar$ is moved from $L$ to $P$, each $\x_i \in L$ needs to be
updated only if the set 
$$
\tilde{A}(\x_i, \xbar) \; = \;
\{ \ba \in \tilde{A}(\x_i) \, \mid \, p(\x_i, \xbar, \ba) > 0 \}
$$
is not empty.
Finally, the new equivalent of the efficient update formula
(\ref{eq:temp_labels_local_Determ}) is now 
\begin{equation}
\label{eq:temp_labels_local_SSP}
U(\x_i) \; := \; \min \left\{ \; U(\x_i), \;
\min\limits_{\ba \in \tilde{A}(\x_i, \xbar)} 
\left\{ 
C(\x_i, \ba) \, + \, 
\sum\limits_{j = 1}^{M+1}  p_{ij}(\ba) U_j 
\right\}
\;
\right\}.
\end{equation}
If a constant $\kappa << M$ is an upper bound on stochastic outdegrees
(i.e., if for $i=1, \ldots, M$ we have $\kappa \geq $ the total
number of nodes $\x_j$ for which $\exists \ba \in A(\x_i)$ such that 
$p_{ij}(\ba) > 0$), then a Dijkstra-like algorithm described above
will have a computational cost of $O(M \log M)$.  Upon its termination,
the resulting $U \in R^M$ will satisfy the system of equations
(\ref{eq:Markov_general}).  The proof of this fact is straight-forward 
and is listed as one of the exercises in \cite{Bertsekas_DPbook}.

\noindent
Here we introduce a similar definition:\\
Given $\delta \geq 0$,
an optimal stationary policy $\mu^*$ is {\em consistently $\delta$-improving} if
\begin{equation}
\label{def:delta_improving}
p_{ij}(\mu^*(\x_i)) > 0 \qquad \Longrightarrow \qquad U_i > U_j + \delta.
\end{equation}
When $\delta > 0$,
it is similarly easy to show that the existence of a consistently 
$\delta$-improving optimal policy guarantees the convergence of a 
Dial-like method with buckets of width $\delta$
to the value function of the SSP.  As in the deterministic case, 
for $\kappa << M$, the resulting cost will be $O(M)$.
Every consistently $\delta$-improving policy is obviously also
consistently improving;
when $\delta = 0$, this reduces to the previous definition.

Unfortunately, conditions (\ref{def:improving}) and 
(\ref{def:delta_improving}) are implicit
since no optimal policy is a priori known.
Thus, for a general SSP applicability of label-setting methods
is hard to check in advance. It is preferable (and more practical)
to develop conditions based on
functions $C(\x_i, \ba)$ and $p(\x_i, \x_j, \ba)$,
which would guarantee that {\em every optimal policy} 
is consistently improving (or $\delta$-improving).
We will refer to such SSPs as 
{\em absolutely $\delta$-causal} 
(or simply {\em absolutely causal} when $\delta=0$).
Before developing such explicit conditions for 
a particular class of Multimode-SSPs in section \ref{s:MSSP}, 
we make several remarks about the general case.

\begin{remark}[{\em Consistently almost improving policies}] \mbox{ }\\
\label{rem:almost_improving}
Comparing the deterministic causality condition (\ref{cond:causality_determ})
with the condition (\ref{def:improving}), it might seem that a Dijkstra-like 
method should work whenever there exists a ``consistently {\em almost} improving
optimal policy'' , i.e., an optimal $\mu^*$ such that
$$
p_{ij}(\mu^*(\x_i)) > 0 \qquad \Longrightarrow \qquad U_i \geq U_j.
$$
A simple example in Figure \ref{fig:_infinite_VI} demonstrates that this is false.
Indeed, for this example a Dijkstra-like method would terminate with
$U_1 = U_2 = +\infty$
even though the optimal policy is consistently  almost improving
and the correct value function is $U_1 = U_2 = 2C$. 
\end{remark}

\begin{remark}[{\em Lower bounds on control cost}] \mbox{ }\\
\label{rem:cost_lower_bound}
If $\mu^*$ is an optimal policy and 
$\ba^* = \mu^*(\x_i)$, then 
$$
C(\x_i, \ba^*) \; = \;
U_i - \sum\limits_{j = 1}^{M+1}  p_{ij}(\ba^*) U_j 
 \; = \;
\sum\limits_{j = 1}^{M+1}  p_{ij}(\ba^*) (U_i - U_j).
$$ 
This means that $C(\x_i, \ba^*) > \delta \geq 0$ when 
$\mu^*$ is $\delta$-improving.
Thus, when building label-setting methods,
the natural class of problems to focus on
is an SSP with 
$$
\Assume{2'} \hspace*{4cm}
C(\x_i, \ba) \, > \, 0 \qquad \text{for all }
\x_i \in \XX \text{ and } \forall \ba \in A(\x_i).
$$
We note that \Assume{2'} and the compactness of all $A_i$'s imply \Assume{2}.
\end{remark}

\begin{remark}[{\em Label-setting on a reachable subgraph}] \mbox{ }\\
\label{rem:reachable}
Consider a reachable set $X_c$ consisting of all nodes $\x \in X$
such that there exists a policy $\pi$ leading from $\x$ to $\T$
with probability 1.
Assumption \Assume{1} states that $X = X_c$.  If this is not the case,
but the condition \Assume{2'} holds, then
$U(\x_i) = +\infty$ for all $\x_i \not \in X_c$.
If a stationary policy $\mu^*$ is optimal, then \Assume{2'}
implies $p_{ij}(\mu^*(\x_i)) = 0$ whenever 
$\x_i \in X_c$ and $\x_j \not \in X_c.$
If $\mu^*$ also satisfies (\ref{def:improving}) on $X_c$, 
then a Dijkstra-like method is still applicable.  Upon its termination,
the value function will be computed correctly on $X_c$ and
we will have $U(\x) = +\infty$ for all $\x \not \in X_c$.
(This is analogous to using the original
Dijkstra's method on a digraph that does not contain directed paths
to $\T$ from every $\x \in X$.)  Of course, an efficient implementation 
will terminate the method as soon as all nodes remaining in $L$
have a label of $+\infty$.
\end{remark}

\begin{remark}[{\em Label-setting for SSP: prior work.}] \mbox{ }\\
\label{rem:LS_prior_work}
It is natural to look for classes of SSPs, for which
either (\ref{def:improving}) or (\ref{def:delta_improving})
is automatically satisfied by every optimal policy.  
One simple example is the deterministic 
case:  if for every $\x_i \in \XX$ and 
$\forall \ba \in A(\x_i)$
there exists $\x_j \in X$ such that $p_{ij}(\ba) = 1$,
then every optimal policy is consistently improving due to \Assume{2'}.
Tsitsiklis was the first to prove causality 
of two truly stochastic SSPs  
\cite{Tsitsiklis_conference, Tsitsiklis}, which he used to
develop Dijkstra-like and Dial-like methods for two special 
discretizations of the Eikonal PDE on a uniform Cartesian grid.  
For Eikonal PDEs discretized on arbitrary acute meshes, the equivalent 
of property (\ref{def:improving}) for all optimal controls was proven in 
\cite[Appendix]{SethVlad3}.  Another implementation of a Dial-like 
method for the Eikonal PDE was introduced in \cite{KimGMM}.
For the optimal control of hybrid systems,
a similar property was used to build Dijkstra-like methods in
\cite{Branicky1} and \cite{SethVladHybrid}.  
It is interesting to note that of all these papers only Tsitsiklis'
work mentions the SSP interpretation of the discretizations, but even 
in \cite{Tsitsiklis} the proof of causality is very problem-specific
and relies on the properties of the PDE and on a particular choice of 
the computational stencil.  In section \ref{s:HJB_discr} we use MSSPs 
to provide convergence criteria for Dijkstra's method in the above cases
as well as the bucket-width for Dial's method whenever it applies.
\end{remark}

\begin{remark}[{\em Label-correcting methods for SSP.}] \mbox{ }\\
\label{rem:LC_SSP}
Whenever the value iteration converges after finitely many steps,
{\em label-correcting methods} become another viable alternative.
Their implementation for the deterministic case can be found in standard 
references (e.g., \cite{Ahuja}, \cite{Bertsekas_NObook}).
Two such methods were introduced in \cite{PolyBerTsi}
for the SSP considered in \cite{Tsitsiklis}.
In a more recent work \cite{BorRasch}, a similar method
was applied to a finite element discretization of the
Hamilton-Jacobi-Bellman PDE.
In the latter case, the label-setting is used 
to obtain convergence-up-to-specified-tolerance even though
the equivalent of condition (\ref{def:improving}) is not satisfied.
Label-setting methods have an optimal worst-case computational cost;
however, in practice label-correcting methods can outperform them on 
many problems.  The exact conditions under which this happens are
still a matter of debate even in deterministic problems.
While clearly interesting, the comparison of their performance on 
SSPs is outside the scope of the current paper.
\end{remark}

\section{Multimode Stochastic Shortest Path Problems.}
\label{s:MSSP}

We will use $\Xi_n$ to denote the set of possible barycentric coordinates in
$\R^n$, i.e., 
$$\Xi_n = \left\{ 
\xi = (\xi_1, \cdots, \xi_n) \, \mid \, 
\xi_1 + \cdots + \xi_n = 1
\text { and } \forall \xi_j \geq 0 
\right\}.$$
We will further define $I(\xi) = \{ i \, \mid \, \xi_i > 0 \}$ and use
$\{ \e_1, \ldots, \e_n \}$ to denote the standard canonical basis in $\R^n$.
Finally, we will use $\R^n_{+,0}$ to denote the non-negative orthant in $\R^n$,
i.e., $\R^n_{+,0} = 
\left\{ 
(x_1, \cdots, x_n) \, \mid \, 
\forall x_j \geq 0 
\right\}.$

We will assume the following
\begin{enumerate}
\item
For every node $\x_i \in \XX$ we are given a list of ``modes'' 
$\M_i = \M(\x_i) = \{m_1, \cdots, m_{r_i} \}$,
where each mode $m \in \M_i$ is a non-empty subset of $X \backslash \{ \x_i \}$ 
and $r_i = r(\x_i) = \left| \M_i \right| \geq 1.$
\item
The nodes within each mode are ordered;
i.e., $m = (\z^m_1, \cdots, \z^m_{|m|})$, where $\z^m_j \neq \z^m_k$ if $j \neq k$.
\item
All controls have a special structure $\ba = (m, \xi \in \Xi_{|m|})$
and there exists an available control 
$(m, \xi) \in A(\x_i)$ for all $m \in \M_i$ and all $\xi \in \Xi_{|m|}$.
\item
The corresponding transition probability is
$$
p \left( \x_i, \x, (m, \xi) \right) \, = \,
\begin{cases}
\xi_j, & \text{ if $\x = z^m_j$ for some $j \in \{1, \cdots, |m|\};$}\\
0, & \text{ otherwise.}
\end{cases}
$$
\item
The transition costs are defined for each mode separately, i.e.,
$C \left (\x_i, (m, \xi) \right) = C^m(\x_i, \xi).$
\item
For $\forall \x_i \in \XX$ and $\forall m \in \M_i$ the
function $C^m(\x_i, \xi)$ is a positive continuous 
function of $\xi$.
\item
There exists a constant upper bound 
$\kappa$ on stochastic outdegrees;\\ 
i.e.,
$\left( \sum\limits_{m \in \M_i} |m| \right) \leq \kappa$
for $i=1, \ldots, M.$
\end{enumerate}

For these MSSPs
it is natural to interpret the decision made at 
each stage as a deterministic choice of a mode $m$
plus the choice of a desirable probability distribution for 
the transition to one of the successor nodes in $m$.
We note that the above framework is sufficiently flexible:
each node can have its own number of modes,
each mode can have its own number of successor nodes,
and different modes can have overlaps (e.g., $\z^{m_1}_j$
can be the same as $\z^{m_2}_k$).  The fully deterministic case
is conveniently included when $|m|=1$ for each mode $m$.  

The above assumptions imply \Assume{0} and \Assume{2'}; hence, the value
iteration converges at least on the reachable subgraph $X_c$ 
(see remarks \ref{rem:cost_lower_bound} and \ref{rem:reachable}).

The dynamic programming equations (\ref{eq:Markov_general})
can be now rewritten as
\begin{eqnarray}
\label{eq:MSSP_DP}
U(\x) & = &  \min\limits_{m \in \M(\x)} 
\left\{ V^m(\x) \right\},\\
\label{eq:mode_DP}
V^m(\x) & = & \min\limits_{\xi \in \Xi_{|m|}} 
\left\{ 
C^m(\x, \xi) \, + \, 
\sum\limits_{j = 1}^{|m|}  \xi_j U(\z^m_j) 
\right\}
.
\end{eqnarray}
Before developing criteria for solvability of the above equations by
label-setting methods (subsections \ref{ss:MSSP_causality} and 
\ref{ss:cost_criteria})
we provide a number of representative examples 
to illustrate the MSSP framework.

\subsection{MSSPs and Modeling.}
\label{ss:MSSP_modeling}

In this subsection we list several examples of 
discrete stochastic control problems, 
which are naturally modeled in the MSSP framework.
Our goal is twofold: to explore the type of stochasticity present in MSSPs
and to understand which types of MSSPs make the development of 
label-setting methods worthwhile.

We begin by considering two very simple MSSPs,
which illustrate the difference and relationship between
explicit and absolute causalities.

\begin{example} 
\label{ex:aux1}
For $M=3$, suppose that each node has only one mode,
and nodes $\x_1,\x_3$ have only one node $\T$ in their modes.
I.e., the transition to $\T$ is deterministic and costs
$C_{it} > 0$ for $i=1,3$.  The $\x_2$'s only mode 
is $m = \{ \x_1, \x_3 \}$.  (See Figure \ref{fig:MSSP_elementary}A.)
Since the problem is so 
simple, it is clear that 
$$
U_1 = C_{1t}; \qquad U_3 = C_{3t};
 \qquad 
 U_2 = \min\limits_{\xi \in \Xi_2} 
\left\{ 
C^m(\x_2, \xi) \, + \, 
(\xi_1 U_1 + \xi_2 U_3)
\right\}.
$$
\end{example}
This SSP is obviously {\em explicitly} causal:
$U_2$ will be computed correctly, provided it is 
computed after $U_1$ and $U_3$
(see Remark \ref{rem:explicit_causality}).\\
However, whether or not this SSP is {\em absolutely} causal
depends on the cost function:\\
Suppose $\xi^*$ is the unique minimizer of the above 
and $C^m$ is such that
$$
U_1 < U_2 < C^m \left( \x_2, 
\left[
\begin{array}{c}
1\\
0
\end{array}
\right]
\right) 
+ U_1 < U_3.
$$
If $\xi^*_2 > 0$, it is
clear that the Dijkstra-like method of section \ref{ss:LS_general_SSP}
would compute $U_2$ incorrectly (since $\x_2$ would be moved from
$L$ to $P$ before $\x_3$).
If label-setting methods were to be used here, we would need to find conditions 
on $C^m(\x_2, \xi)$, which make the above scenario impossible for any
choice of positive $C_{1t}$ and $C_{3t}$.

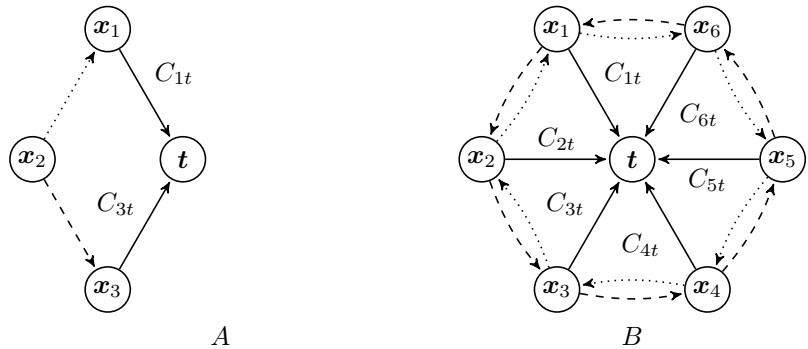
\begin{figure}[hhhh]
\center{
$
\begin{array}{cc}
\begin{tikzpicture}[->,>=stealth',shorten >=1pt,auto,node distance=3cm, scale=2,
                    semithick]
  \tikzstyle{every state}=[draw, shape=circle, inner sep=0mm, minimum size = 6mm]

  \node[state]         (T)	at (0,0)					{$\T$};
  \node[state]         (x1)	at (-0.5, 0.866)			{$\x_1$};
  \node[state]         (x2)	at (-1,0)					{$\x_2$};
  \node[state]         (x3)	at (-0.5, -0.866)			{$\x_3$};

  \path (x1)	edge								node {$C_{1t}$}		(T);
  \path (x2)	edge	[dashed]					node {}				(x3)
				edge	[dotted]					node {}				(x1);
  \path (x3)	edge								node {$C_{3t}$}		(T);
\end{tikzpicture}
\hspace*{3cm}
&
\begin{tikzpicture}[->,>=stealth',shorten >=1pt,auto,node distance=3cm, scale=2,
                    semithick]
  \tikzstyle{every state}=[draw, shape=circle, inner sep=0mm, minimum size = 6mm]

  \node[state]         (T)	at (0,0)					{$\T$};
  \node[state]         (x5)	at (1,0)					{$\x_5$};
  \node[state]         (x6)	at (0.5, 0.866)				{$\x_6$};
  \node[state]         (x1)	at (-0.5, 0.866)			{$\x_1$};
  \node[state]         (x2)	at (-1,0)					{$\x_2$};
  \node[state]         (x3)	at (-0.5, -0.866)			{$\x_3$};
  \node[state]         (x4)	at (0.5, -0.866)			{$\x_4$};

  \path (x1)	edge	[bend right=10, dashed]		node {}				(x2)
				edge	[bend right=10, dotted]		node {}				(x6)
				edge								node {$C_{1t}$}		(T);
  \path (x2)	edge	[bend right=10, dashed]		node {}				(x3)
				edge	[bend right=10, dotted]		node {}				(x1)
				edge								node {$C_{2t}$}		(T);
  \path (x3)	edge	[bend right=10, dashed]		node {}				(x4)
				edge	[bend right=10, dotted]		node {}				(x2)
				edge								node {$C_{3t}$}		(T);
  \path (x4)	edge	[bend right=10, dashed]		node {}				(x5)
				edge	[bend right=10, dotted]		node {}				(x3)
				edge								node {$C_{4t}$}		(T);
  \path (x5)	edge	[bend right=10, dashed]		node {}				(x6)
				edge	[bend right=10, dotted]		node {}				(x4)
				edge								node {$C_{5t}$}		(T);
  \path (x6)	edge	[bend right=10, dashed]		node {}				(x1)
				edge	[bend right=10, dotted]		node {}				(x5)
				edge								node {$C_{6t}$}		(T);
\end{tikzpicture} \\
A & B
\end{array}
$
}
\caption{
{\footnotesize
Two simple examples of MSSP.  In both cases, starting from $\x_i$, 
one needs to select an optimal probability distribution over two 
successor nodes (dashed \& dotted lines) or to opt for the
deterministic transition to $\T$ (priced at $C_{it}>0$ and 
shown by solid lines wherever available).
}
}
\label{fig:MSSP_elementary}
\end{figure}

\begin{example} 
\label{ex:circular_list}
For a somewhat more interesting example, consider 
a circular doubly linked list of $M$ nodes.
(See Figure \ref{fig:MSSP_elementary}B for the case $M=6$.)
Each $\x_i$ has two modes: $m = (\x_{prev}, \x_{next})$
and $m' = \{ \T \}$.  
\end{example}
The applicability of label-setting methods seems 
harder to judge in this case, but 
it is clear that we don't want $\x_i$ moved from $L$ to $P$
before its neighbors if the optimal choice at $\x_i$
involves a possible (non-deterministic) transition to one of them.
For instance, 
$$
U_2 = \min \; \left\{ \, C_{2t}, \; 
\min\limits_{\xi \in \Xi_2} 
\left\{ 
C^m(\x_2, \xi) \, + \, 
(\xi_1 U_1 + \xi_2 U_3)
\right\}
\, \right\},
$$
and we note that this equation is provably absolutely causal if
in Example \ref{ex:aux1}
a Dijkstra-like method produces correct $U_2$  
for all allowable $C_{1t}$ and $C_{3t}$.
In fact, if the same $C^m(\x_i, \xi)$ is used for each $\x_i$,
the above condition is sufficient to show
the absolute causality of the full problem.  
This idea is generalized in subsection \ref{ss:MSSP_causality}.

In practical terms, whether or not the MSSP in Example \ref{ex:aux1}
is absolutely causal is irrelevant since the value function can be 
easily computed directly (see Remark \ref{rem:explicit_causality}).  
On the other hand,
Example \ref{ex:circular_list} can be viewed as a variant of an optimal 
stopping problem, whose absolute causality would yield a more efficient 
alternative to the basic value iteration when $M$ is large.
We continue by considering a number of interesting 
single-mode-for-each-node examples.

In the opening act of Tom Stoppard's famous play \cite{Stoppard},
the title characters engage in statistical experimentation with 
supposedly fair coins.  The fairness of their coins is highly suspect since they
are observing a very long and uninterrupted run of ``heads''.
Rosencrantz (Ros) is bored by the game and would be glad to stop playing, 
but Guildenstern (Guil) insists on continuing.
The following two examples are inspired by the above.  

\begin{example} 
\label{ex:GR_1}
Suppose Guil will agree to stop only
after observing $K$ ``heads'' in a row.  Ros has to pay some fee for every toss of a coin and 
is interested in minimizing his expected total cost up to the termination.  
Moreover, suppose that for each toss Ros can request a coin with any probability distribution 
$(p, (1-p))$ on possible outcomes (``heads'' vs. ``tails''), but Guil intends to charge him  
a different fee $C(p)$ based on his request.  The problem is to find an optimal 
$p_i^* \in [0,1]$ that Ros should request after observing $i$ ``heads'' in a row
(i.e., in the state $\x_i$). 
\end{example}
Figure \ref{fig:RosGuil_1} (Left) shows the graph representation of the game for $K=3$.
Denoting $\x_K = \T$, we set $U_K = 0$.
Since there is exactly one mode per node,
and two successor-nodes only ($\xi \in \Xi_2; \, \xi_1 = p, \, \xi_2 = (1-p)$),
the Dynamic Programming equations of this game can be re-written as
$$
U_i 
\; =  \;
\min\limits_{\xi \in \Xi_2} 
\left\{ 
C(\x_i, \xi) \, + \, 
\xi_1 U_{i+1} + \xi_2 U_0
\right\}
\; =  \;
\min\limits_{p \in [0,1]} 
\left\{ 
C(p) \, + \, 
p U_{i+1} + (1-p) U_0
\right\}.
\qquad
\text{ for } i=0, \ldots, K-1. 
$$

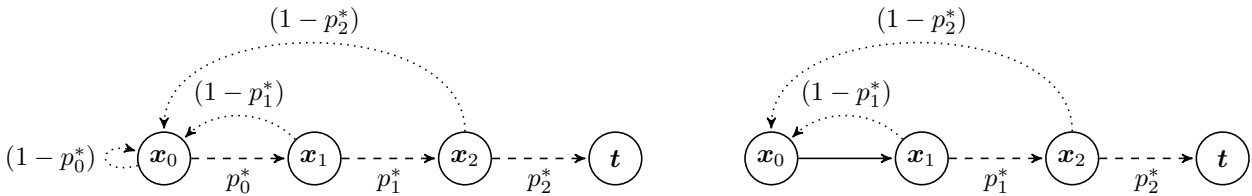
\begin{figure}[hhhh]
\center{
$
\hspace*{-10mm}
\begin{array}{lr}
\begin{tikzpicture}[->,>=stealth',shorten >=1pt,auto,node distance=2cm, 
                    semithick]
  \tikzstyle{every state}=[draw, shape=circle, inner sep=0mm, minimum size = 7mm]

  \node[state]         (x0)	at (0,0)				{$\x_0$};
  \node[state]         (x1) [right of=x0]			{$\x_1$};
  \node[state]         (x2) [right of=x1]			{$\x_2$};
  \node[state]         (T)	[right of=x2]			{$\T$};
  
  \path (x0)		edge	[dashed]					node [swap]		{$p_0^*$}			(x1)
					edge	[loop left, dotted]		node [swap]		{$(1-p_0^*)$}		(x0);

  \path (x1)		edge	[dashed]					node [swap]		{$p_1^*$}			(x2)
					edge	[bend right=45, dotted]		node [swap]		{$(1-p_1^*)$}		(x0);

  \path (x2)		edge	[dashed]					node [swap]		{$p_2^*$}			(T)
					edge	[bend right=90, dotted]		node [swap]		{$(1-p_2^*)$}		(x0);
\end{tikzpicture}
&
\hspace*{1cm}
\begin{tikzpicture}[->,>=stealth',shorten >=1pt,auto,node distance=2cm,
                    semithick]
  \tikzstyle{every state}=[draw, shape=circle, inner sep=0mm, minimum size = 7mm]

  \node[state]         (x0)	at (0,10)				{$\x_0$};
  \node[state]         (x1) [right of=x0]			{$\x_1$};
  \node[state]         (x2) [right of=x1]			{$\x_2$};
  \node[state]         (T)	[right of=x2]			{$\T$};
  
  \path (x0)		edge								node			{}				(x1);

  \path (x1)		edge	[dashed]					node [swap]		{$p_1^*$}			(x2)
					edge	[bend right=45, dotted]		node [swap]		{$(1-p_1^*)$}		(x0);

  \path (x2)		edge	[dashed]					node [swap]		{$p_2^*$}			(T)
					edge	[bend right=90, dotted]		node [swap]		{$(1-p_2^*)$}		(x0);
\end{tikzpicture}
\end{array}
$
}
\caption{
{\footnotesize
The first Guildenstern and Rosencrantz game for $K=3$ (Left).  
After $i$ ``heads'' in a row the game-state is $\x_i$.
Transitions corresponding to ``heads'' and 
``tails'' are shown by dashed and dotted lines respectively.  The self-transition in $\x_0$ can be
removed and replaced by a deterministic transition (solid line) with the optimal cost $C_{01}$ (Right).
}
}
\label{fig:RosGuil_1}
\end{figure}

We note that the self-transition in the node $\x_0$ can be dealt with in the spirit of 
Remark \ref{rem:self_transition}; see Figure \ref{fig:RosGuil_1} (Right).  
This results in a deterministic transition to $\x_1$:
$$
U_0 = C_{01} + U_1,
\qquad 
\text{where }
C_{01} = \min\limits_{p \in (0,1]} \frac{C(p)}{p} = \frac{C(p^*_0)}{p^*_0}.
$$
After this simplification, the example satisfies all the assumptions listed for MSSPs;
therefore, the applicability of label-setting methods can be determined by checking if
$C(\xi)$ satisfies any of the criteria developed in section \ref{ss:cost_criteria}.

\begin{remark}
Even though this MSSP is not explicitly causal,
a simple structure of the graph makes it almost trivial for our purposes:\\
1.  Every (eventually terminating) path from
$\x_i$ leads through $\x_{i+1}$, which implies $U_i > U_{i+1}$.  Thus, label-setting methods
are only applicable if $p^*_i = 1$ for all $i$.\\  
2.  On the other hand, recursive relations similar to the one used above can be 
repeatedly employed to make this into a deterministic problem.  For example,
\begin{eqnarray*}
U_1
& =&
\min\limits_{p \in [0,1]} 
\left\{ 
C(p) \, + \, 
p U_2 + (1-p) U_0
\right\}
=
\min\limits_{p \in [0,1]} 
\left\{ 
C(p) \, + \, 
p U_2 + (1-p) (C_{01} + U_1)
\right\}\\
&=&
\min\limits_{p \in [0,1]} 
\left\{ 
\left[C(p) + (1-p) C_{01}\right] \, + \, 
p U_2 + (1-p) U_1
\right\}
=
C_{12} + U_2,
\end{eqnarray*}
where
$$
C_{12} = \min\limits_{p \in (0,1]} \frac{C(p) + (1-p) C_{01}}{p} = \frac{C(p^*_1) + (1-p^*_1) C_{01}}{p^*_1}.
$$
Repeating this procedure we can compute the value function in $O(K)$ steps 
(counting the above minimization
as a single operation) even if some of the $p^*_i$'s are less than one
(in which case the value iteration would not converge in a finite number of steps).
\label{rem:R_and_G_boring}
\end{remark}

\begin{example} 
\label{ex:GR_2}
Now suppose that Guil will agree to stop only
after observing an uninterrupted run of $K_h$ ``heads''  or $K_t$ ``tails'';
see Figure \ref{fig:RosGuil_2}.
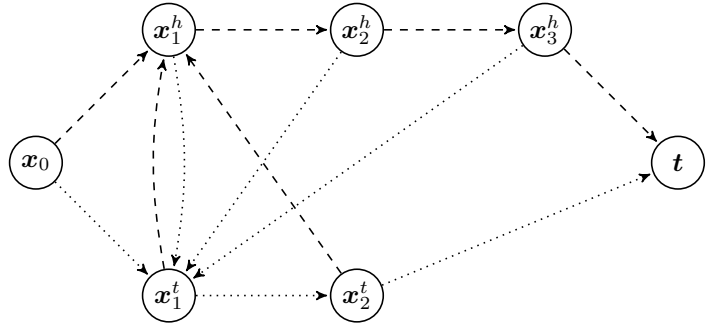
\begin{figure}[hhhh]
\begin{center}
\begin{tikzpicture}[->,>=stealth',shorten >=1pt,auto,node distance=2.5cm,
                    semithick]
  \tikzstyle{every state}=[draw, shape=circle, inner sep=0mm, minimum size = 7mm]

  \node[state]         (x0)							{$\x_0$};
  \node[state]         (x1h) [above right of=x0]	{$\x^h_1$};
  \node[state]         (x2h) [right of=x1h]			{$\x^h_2$};
  \node[state]         (x3h) [right of=x2h]			{$\x^h_3$};
  \node[state]         (T)	[below right of=x3h]	{$\T$};
  \node[state]         (x1t) [below right of=x0]	{$\x^t_1$};
  \node[state]         (x2t) [right of=x1t]			{$\x^t_2$};
 
  \path (x0)		edge	[dashed]					node			{}					(x1h);
  \path (x1h)		edge	[dashed]					node			{}					(x2h);
  \path (x2h)		edge	[dashed]					node			{}					(x3h);
  \path (x3h)		edge	[dashed]					node			{}					(T);
  \path (x0)		edge	[dotted]					node			{}					(x1t);
  \path (x1t)		edge	[dotted]					node			{}					(x2t);
  \path (x2t)		edge	[dotted]					node			{}					(T);

  \path (x1h)		edge	[dotted, bend left=10]		node			{}					(x1t);
  \path (x2h)		edge	[dotted]					node			{}					(x1t);
  \path (x3h)		edge	[dotted]					node			{}					(x1t);

  \path (x1t)		edge	[dashed, bend left=10]		node			{}					(x1h);
  \path (x2t)		edge	[dashed]					node			{}					(x1h);
\end{tikzpicture}
\end{center}
\caption{
{\footnotesize
The second Guildenstern and Rosencrantz game for $K_h=4$ and $K_t=3$.
After $i$ ``heads'' or ``tails'' in a row the game-state is $\x_i^h$ or $\x_i^t$ respectively.
}
}
\label{fig:RosGuil_2}
\end{figure}
\end{example}
Identifying $\T = \x^h_{K_h} = \x^t_{K_t}$ and $\x_0 = \x^h_0 = \x^t_0$, 
we can re-write the dynamic programming equations as
\begin{eqnarray*}
U(\T) &=& 0;\\
U(\x^h_i) &=& 
\min\limits_{p \in [0,1]} 
\left\{ 
C(p) \, + \, 
p U(\x^h_{i+1}) + (1-p) U(\x^t_1)
\right\};
\qquad
\text{ for } i=0, \ldots, K_h-1;\\
U(\x^t_i) &=& 
\min\limits_{p \in [0,1]} 
\left\{ 
C(p) \, + \, 
p U(\x^h_1) + (1-p) U(\x^t_{i+1})
\right\};
\qquad
\text{ for } i=0, \ldots, K_t-1.
\end{eqnarray*}
Since the Remark \ref{rem:R_and_G_boring} does not apply, 
in this case it is possible
to have a non-trivial optimal strategy (i.e., $p^* \in (0,1)$),
which might be computable by the label-setting methods.
Their applicability can be guaranteed by certain properties of 
the cost function
as will be shown by theorems of section \ref{ss:cost_criteria}.
For example, this MSSP is absolutely causal
(and thus efficiently computable using a Dijkstra-like method
regardless of specific values of $K_h$ and $K_t$)
for 
$$
C_1(p) =  3 + 2 p - p^4 - (1-p)^2;
\qquad \text{or} \qquad
C_2(p) =  \sqrt{p^2 + (1-p)^2};
\qquad \text{or} \qquad
C_3(p) =  4 + (p-0.5)^3.
$$
(Recalling that $\xi_1 = p$ and $\xi_2 = (1-p)$,
it will be easy to check that Theorems \ref{thm:concave}, \ref{thm:homogen}, 
and \ref{thm:second_deriv_bound}
apply to $C_1$, $C_2$, and $C_3$ respectively.)

A similar analysis works
when Guil is allowed to use different prices
depending on the current state of the game (i.e., with $C(\x_i, p)$ instead of $C(p)$)
or 
when the number of possible outcomes is higher (e.g., dice instead of coins, $\Xi_6$
instead of $\Xi_2$.)

\begin{example} 
\label{ex:multitask}
Suppose a person is engaged in multi-tasking, dividing her attention between
activities A and B.  This allocation of resources is described by 
$\xi = (\xi_A, \xi_B) \in \Xi_2$.
We assume that\\ 
$\bullet \,$
per every time-unit she reaches a new milestone in 
{\em exactly one} of these activities;\\
$\bullet \,$
the probability of a milestone reached in A or B is proportional to the fraction of 
her attention invested in that activity ($\xi_A$ or $\xi_B$) during that time-unit;\\
$\bullet \,$
the current state of the process $\x_{i,j}$ reflects the number of milestones reached 
in both activities;\\
$\bullet \,$
the cost (per time-unit) of all possible resource allocations 
is specified by $C(\x_{i,j}, \xi)$;\\
$\bullet \,$
the process terminates after at least $K_A$ milestones are reached in A or
at least $K_B$ milestones reached in B;\\
$\bullet \,$
the goal is to minimize the total expected cost up to a termination.\\
A particular instance of this problem is illustrated in Figure \ref{fig:miltitask}.

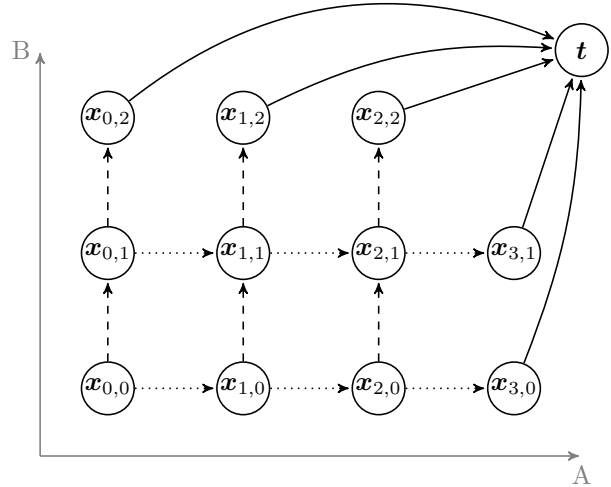
\begin{figure}[hhhh]
\begin{center}
\begin{tikzpicture}[->,>=stealth',shorten >=1pt,auto,node distance=2.5cm,
                    semithick, scale=1.8]
  \tikzstyle{every state}=[draw, shape=circle, inner sep=0mm, minimum size = 7mm]

\draw[gray] (-0.5,-0.5) -> (3.5,-0.5) node [below] {A};
\draw[gray] (-0.5,-0.5) -> (-0.5,2.5) node [left] {B};

  \node[state]         (x00)	 at (0,0)				{$\x_{0,0}$};
  \node[state]         (x01)	 at (0,1)				{$\x_{0,1}$};
  \node[state]         (x02)	 at (0,2)				{$\x_{0,2}$};
  \node[state]         (x10)	 at (1,0)				{$\x_{1,0}$};
  \node[state]         (x11)	 at (1,1)				{$\x_{1,1}$};
  \node[state]         (x12)	 at (1,2)				{$\x_{1,2}$};
  \node[state]         (x20)	 at (2,0)				{$\x_{2,0}$};
  \node[state]         (x21)	 at (2,1)				{$\x_{2,1}$};
  \node[state]         (x22)	 at (2,2)				{$\x_{2,2}$};
  \node[state]         (x30)	 at (3,0)				{$\x_{3,0}$};
  \node[state]         (x31)	 at (3,1)				{$\x_{3,1}$};

  \node[state]         (t)	 at (3.5,2.5)					{$\T$};

  \path (x00)		edge	[dashed]	node	{}		(x01)
					edge	[dotted]	node	{}		(x10);
  \path (x10)		edge	[dashed]	node	{}		(x11)
					edge	[dotted]	node	{}		(x20);
  \path (x20)		edge	[dashed]	node	{}		(x21)
					edge	[dotted]	node	{}		(x30);
  \path (x30)		edge	[bend right=10]			node	{}		(t);

  \path (x01)		edge	[dashed]	node	{}		(x02)
					edge	[dotted]	node	{}		(x11);
  \path (x11)		edge	[dashed]	node	{}		(x12)
					edge	[dotted]	node	{}		(x21);
  \path (x21)		edge	[dashed]	node	{}		(x22)
					edge	[dotted]	node	{}		(x31);
  \path (x31)		edge	[]			node	{}		(t);

  \path (x02)		edge	[bend left=30]	node	{}		(t);
  \path (x12)		edge	[bend left=15]	node	{}		(t);
  \path (x22)		edge	[]				node	{}		(t);

\end{tikzpicture}
\end{center}
\caption{
{\footnotesize
The multitasking problem for $K_A=3$ and $K_B=2$.
Each of the nodes $\x_{i,K_B}$ or $\x_{K_A,j}$ 
has a deterministic transition to $\T$ only.
All other nodes $\x_{i,j}$ have a single mode
$(\x_{i,j+1}, \x_{i+1,j})$.  
The node $\x_{K_A,K_B}$ is not needed since the 
process always terminates before reaching it.
}
}
\label{fig:miltitask}
\end{figure}
\end{example}
The above MSSP is obviously explicitly causal
since the number of milestones achieved in each activity
can only increase as time goes on regardless of the chosen policy.
As usual with explicitly causal SSPs,
the causal ordering of the nodes is a priori known regardless of 
the cost functions 
and the label-setting methods are really not needed.
However, a slight variation of the above is already computationally 
challenging:
\begin{example} 
\label{ex:multitask_strong}
Suppose the same person also dedicates a part of her
attention to some distraction D and her
resource allocation is now $\xi = (\xi_A, \xi_B, \xi_D) \in \Xi_3$,
where $\xi_D$ is the probability of getting completely distracted
and inadvertently ``resetting'' the process 
(i.e., transition into $\x_{0,0}$).
\end{example}
If the diversion is appealing (i.e., if 
$C(\x_{i,j}, \xi)$ is a decreasing function of $\xi_D$), this
problem is not explicitly causal and the applicability of
label-setting methods becomes relevant.
The possibility of self-transition in $\x_{0,0}$ is again 
dealt with in the spirit of Remark \ref{rem:self_transition}
and theorems from section \ref{ss:cost_criteria} can be then used 
to test for the absolute causality.
Generalizations of this example (to an arbitrary number of activities
and/or partial resets due to a diversion) can be handled similarly. 

We note that the MSSPs occupy a niche in between purely deterministic
and generally stochastic shortest path problems.
It is easy to see that in all of the above examples 
the stochastic aspect of the model is not due to some uncontrollable event
(after all, the deterministic/pure controls are always available in MSSPs),
but rather due to our belief that a randomized/mixed control might 
carry a lower cost.  

\begin{remark}[{\em Randomized/mixed controls \& deterministic SP}] \mbox{ }\\
\label{rem:randomized_mixed}
In most deterministic discrete control problems 
mixed policies are considered unnecessary.
But this is mainly due to the fact that 
the cost of implementing such mixed/randomized controls
is usually modeled by a linear function,
i.e., 
$  
C^m(\x, \xi) = \sum\limits_{j = 1}^{|m|}  \xi_j C_j.
$
(More generally, Theorem \ref{thm:concave} will show that an optimal
control can be found among the pure controls $\{\e_j\}$
 for any concave cost function.)
However, if the cost is non-concave, i.e., if
$  
C^m(\x, \xi) < \sum\limits_{j = 1}^{|m|}  \xi_j C^m(\x, \e_j)
$
for at least some $\xi \in \Xi_{|m|}$,
then a mixed strategy is available ``at a discount'' and might 
be advantageous.
\end{remark}

\noindent 
The methods developed in this paper are therefore most useful
for MSSPs that\\
$\bullet \,$
are not explicitly causal (otherwise direct methods are more efficient);\\
$\bullet \,$
but are absolutely causal due to (possibly non-concave) costs satisfying criteria 
in Section \ref{ss:cost_criteria}.

\noindent
Additional examples (stemming from discretizations of continuous
optimal control problems) are discussed in section \ref{s:HJB_discr}.

\subsection{Causality of MSSP and single-mode auxiliary problems.}
\label{ss:MSSP_causality}
\noindent
Checking whether a given MSSP is absolutely causal can be hard,
although sufficient conditions can be developed hierarchically.
This approach was already used in subsection \ref{ss:MSSP_modeling}
to show the relationship between examples \ref{ex:aux1} and \ref{ex:circular_list}.

\noindent
In the general case,
if $\mu^*$ is an optimal policy and $\mu^*(\x) = (m^*, \xi^*)$, then
the formulas (\ref{eq:MSSP_DP} - \ref{eq:mode_DP}) imply
$$
U(\x) = V^{m^*}(\x) = C^{m^*}(\x, \xi^*) \, + \, 
\sum\limits_{j = 1}^{|m^*|}  \xi_j^* U(\z^{m^*}_j).
$$
If $\mu^*$ is consistently $\delta$-improving, we have
$
\quad
(\xi^*_j > 0) \qquad \Longrightarrow \qquad V^{m^*}(\x) > U(\z^{m^*}_j) + \delta.
$

\begin{observ}
\label{obs:implicit_causality}
For each mode $m$ let $\Xi^* \subset \Xi_{|m|}$ be a set of all 
minimizers in formula (\ref{eq:mode_DP}).
If 
\begin{equation}
\label{def:delta_improving_MSSP}
(\xi^*_j > 0) \quad \Longrightarrow \quad V^m(\x) > U(\z^m_j) + \delta,
\qquad \text{for }
\begin{array}{l}
\forall \x \in \XX; \; \forall m \in \M(\x);\\
\forall \xi^* \in \Xi^*; \; j = 1, \ldots, |m|
\end{array}
\end{equation}
then {\em every} optimal policy is consistently $\delta$-improving
(and this MSSP is absolutely $\delta$-causal).
\end{observ} 

As a result, we can develop label-setting applicability conditions
on a mode-per-mode basis.
In the following we will focus on one $\x \in \XX$ and one mode $m \in \M(\x)$ 
to develop conditions on $C^m(\x, \cdot)$ that guarantee causality for 
all possible values of $U(\z^m_j)$'s.  Since $\x$ and $m$ are fixed,
we will simplify the notation by using
$$
V = V^m(\x); \quad C(\cdot) = C^m(\x, \cdot); \quad
W_j = U(\z^m_j); \quad n=|m|.
$$
Furthermore, interpreting $\xi$ and $W$ as column vectors in $\R^n$,
we define $F: \Xi_{n} \times \R^n \mapsto \R$ as follows:
$$
F(\xi, W) \; = \; C(\x, \xi) \, + \, \xi^T \, W.
$$

\noindent
The dynamic programming equation (\ref{eq:mode_DP})
can be now rewritten as 
\begin{equation}
\label{eq:1_mode_DP}
V \; = \; \min\limits_{\xi \in \Xi_n} 
\left\{ 
C(\xi) \, + \, 
\sum\limits_{j = 1}^{n}  \xi_j W_j 
\right\} 
\; = \; 
\min\limits_{\xi \in \Xi_n} 
F(\xi, W)
\end{equation} 

\noindent
Once the vector $W$ is specified, 
this also determines the set of minimizers
$\Xi^*(W)  = \argmin\limits_{\xi \in \Xi_n} F(\xi, W).$

\begin{definition}
\label{def:mode_delta_causal}
The mode $m$ is {\em absolutely $\delta$-causal} if
$$
(\xi^*_j > 0) \quad \Longrightarrow \quad V > W_j + \delta,
\qquad \text{for } \;
\forall W \in \R^n_{+,0}; \ 
\forall \xi^* \in \Xi^*(W); \ j = 1, \ldots, n.
$$
We will also refer to a mode as {\em absolutely causal}
if the above holds at least with $\delta=0$.
\end{definition}

A simple way to interpret this definition
is by considering an auxiliary single-mode MSSP
on the nodes $\{ \x, \z^m_1, \ldots, \z^m_n, \T \}$
with a single mode for each node (see Figure \ref{fig:auxiliary}).
Let the transition from each $\z^m_j$ to
$\T$ be deterministic with cost $C_{jt} = W_j \geq 0$,
and for $\x$ let the mode be $m = (\z^m_1, \ldots, \z^m_n)$
using the transition cost $C^m(\x,\cdot)$ from the original problem.

\begin{figure}[hhhh]
\center{
\begin{tikzpicture}[->,>=stealth',shorten >=1pt,auto,node distance=3cm, scale=1,
                    semithick]
  \tikzstyle{every state}=[draw, shape=circle, inner sep=0mm, minimum size = 6mm]

  \node[state]         (T)	at (3,0)					{$\T$};
  \node[state]         (x)	at (-3,0)					{$\x$};
  \node[state]         (z1)	at (0, 2.4)					{$\z^m_1$};
  \node[state]         (z2)	at (0, 0.8)					{$\z^m_2$};
  \node[state]         (z3)	at (0, -0.8)				{$\z^m_3$};
  \node[state]         (z4)	at (0, -2.4)				{$\z^m_4$};

  \path (x)		edge	[dashed]					node {}				(z1)
				edge	[dashed]					node {}				(z2)
				edge	[dashed]					node {}				(z3)
				edge	[dashed]					node {}				(z4);
  \path (z1)	edge	[near start, above=1pt]				node {$C_{1t}$}		(T);
  \path (z2)	edge	[near start, above=1pt]				node {$C_{2t}$}		(T);
  \path (z3)	edge	[near start, above=1pt]				node {$C_{3t}$}		(T);
  \path (z4)	edge	[near start, above=1pt]				node {$C_{4t}$}		(T);
\end{tikzpicture}
}
\caption{
{\footnotesize
An auxiliary single-mode problem for $m \in \M(\x)$.
Deterministic transition are shown by solid lines; $n = |m| = 4.$
}
}
\label{fig:auxiliary}
\end{figure}
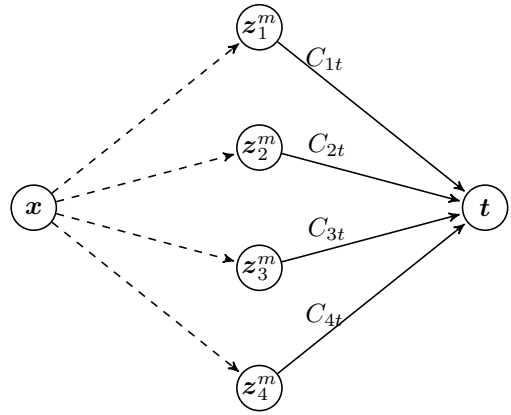

\noindent
The mode $m$ is absolutely causal if a Dijkstra-like method
solves the auxiliary single-mode problem correctly for every vector 
$W \in \R^n_{+,0}.$  The mode is absolutely $\delta$-causal if 
the same is true for a Dial-like method with buckets of width $\delta$.
In fact, Example \ref{ex:aux1} can be viewed as such auxiliary
problem for the mode $m \in \M(\x_2)$ of Example \ref{ex:circular_list}.
We emphasize that the absolute causality of auxiliary problems
is desirable not because we intend to use label-setting on any of them
(after all, each auxiliary problem is explicitly causal, 
and a direct computation is efficient; see Remark \ref{rem:explicit_causality}), 
but because the label-setting methods might be advantageous on the original MSSP.

The conditions on the mode $m$ 
in the definition
\ref{def:mode_delta_causal}
are more restrictive
then those in 
(\ref{def:delta_improving_MSSP}) since in the latter case
the $\delta$-causality is needed for only one (albeit unknown)
vector $W$.  Thus, Observation \ref{obs:implicit_causality}
yields the following sufficient condition
for applicability of label-setting methods to MSSPs:
\begin{corollary}
\label{thm:LS_for_MSSP}
For a general MSSP, if every mode of every node is
absolutely causal, then the MSSP is also absolutely causal and 
a Dijkstra-like method is applicable.
If each mode $m$ is absolutely $\delta_m$-causal,
then the MSSP is absolutely $\Delta$-causal with
$$
\Delta \; = \;
\left(\min\limits_{\x \in \XX, \, m \in \M(\x)} \; \{ \delta_m \} \right)
$$ 
and a Dial-like method is applicable if $\Delta > 0$. 
\end{corollary}
We note that it is possible to have an absolutely causal MSSP
some of whose modes are not absolutely causal.  
This is reminiscent of the fact that the original
Dijkstra's method might be converging correctly even for 
some deterministic problems with negative transition 
penalties.

\subsection{Criteria on cost and absolute causality.}
\label{ss:cost_criteria}
Consider a mode $m \in \M(\x)$ such that $n=|m|>1$.
In view of Corollary \ref{thm:LS_for_MSSP}, it is important
to find additional conditions on the transition cost function
$C(\cdot) = C^m(\x, \cdot)$ that guarantee $m$'s 
absolute $\delta$-causality.

One obvious example is
$  
C(\xi) = \sum\limits_{j = 1}^{n}  \xi_j C_j,
$
where all $C_j$'s are positive constants.
In that case $F$ is linear and equation (\ref{eq:1_mode_DP})
reduces to 
$$
V \; = \; \min\limits_{\xi \in \Xi_n} F(\xi, W)
\; = \; \min\limits_{\xi \in \Xi_n} 
\left\{ 
\sum\limits_{j = 1}^{n}  \xi_j (C_j + W_j) 
\right\} 
\; = \; \min\limits_j 
\left\{ 
( C_j + W_j)
\right\}, 
$$
which is not different from the deterministic shortest path 
equation (\ref{eq:Optimality_Determ}).  The same principle
works for arbitrary concave costs.  

\begin{theorem}
Suppose 
$$
\Assume{3}	\hspace*{6cm}
C: \R^n \mapsto \R_+ \quad \text{ is concave.}
$$
Then the mode $m$ is absolutely $\delta$-causal with
$\delta = \min\limits_{j} \, C(\e_j).$  Moreover,
$V$ can be more efficiently evaluated as
$V \; = \; \min\limits_j 
\left\{ 
\left( C(\e_j) + W_j \right)
\right\}.$ 
\label{thm:concave}
\end{theorem}
\begin{proof}
Since $F(\xi, W) = C(\xi) + \xi^T W$ we know that
the function $F(\xi, W)$ is concave on $\Xi_n$.
Thus, if $\xi^* \in \Xi^*(W) = \argmin\limits_{\xi \in \Xi_n} F(\xi, W)$
then 
$$
(\xi^*_j > 0) \quad \Longrightarrow \quad F(\e_j, W) = F(\xi^*_j, W)
\quad \Longrightarrow \quad U - W_j =  C(\e_j) \geq \delta > 0,
$$
hence the mode $m$ is absolutely $\delta$-causal.
\end{proof}

Homogeneous cost functions naturally arise in many SSPs.
We recall that a function $L(\y)$ is absolutely homogeneous of degree $d$
if $L(a \y) = |a|^d L(\y)$ for all $\y \in \R^n, \, a \in \R.$
If $L$ is also smooth, 
by Euler's Homogeneous Function Theorem, it satisfies the equation
$\y^T \nabla L(\y) = d L(\y)$.

\begin{lemma} Suppose the cost 
$$
\Assume{4} \qquad C: \R_+^n \mapsto \R_+ \text{ is 
continuously differentiable and absolutely homogeneous of degree $d$.}
$$
Then for every $W \in \R_{+,0}^n, \, \xi^* \in \Xi^*(W)$ we have
$$ 
\quad (\xi^*_j > 0)
\quad \Longrightarrow \quad 
V - W_j = \frac{\partial C}{\partial \xi_j}(\xi^*) - (d-1) C(\xi^*).
$$
\label{lemma:homogen}
\end{lemma}
\begin{proof}
For all $j \in I(\xi^*)$ the Kuhn-Tucker optimality conditions
state that 
\begin{equation}
 \label{eq:KT_1}
\lambda \; = \; W_j + \frac{\partial C}{\partial \xi_j}(\xi^*),
\end{equation}
where $\lambda$ is a Lagrange multiplier.  We recall that
$\sum\limits_{j \in I(\xi^*)} \xi^*_j = 1$.
Multiplying (\ref{eq:KT_1}) by $\xi^*_j$ and summing over all
$j \in I(\xi^*)$ we obtain
\begin{equation}
 \label{eq:KT_2}
\lambda = 
\sum\limits_{j \in I(\xi^*)} \xi^*_j \lambda =
\sum\limits_{j \in I(\xi^*)} \xi^*_j
\left(
W_j + \frac{\partial C}{\partial \xi_j}(\xi^*)
\right) = 
\sum\limits_{i=1}^n \xi^*_i W_i + 
\left(
\sum\limits_{i=1}^n \xi^*_i 
 \frac{\partial C}{\partial \xi_i}(\xi^*)
\right).
\end{equation}
Thus, by Euler's Homogeneous Function Theorem,
\begin{equation}
\label{eq:KT_3}
\lambda = \sum\limits_{i=1}^n \xi^*_i W_i + d C(\xi^*) = 
F(\xi^*, W) + (d-1) C(\xi^*).
\end{equation}
Since $\xi^*$ is a minimizer, $V = F(\xi^*, W) = \lambda - (d-1) C(\xi^*)$ and
it follows from (\ref{eq:KT_1}) that
$$
V - W_j = \frac{\partial C}{\partial \xi_j}(\xi^*) - (d-1) C(\xi^*)
\qquad \text{for all } W \in R^n_{+,0}, \, \xi^* \in \Xi^*(W), j \in I(\xi^*).
$$
\end{proof}

\begin{theorem}
If $C$ satisfies \Assume{4} and 
$$
\Assume{5} \hspace*{3cm}
\frac{\partial C}{\partial \xi_j}(\xi) - (d-1) C(\xi) > \delta \geq 0
\qquad \text{for } \forall \xi \in \Xi_n, \forall j \in I(\xi).
$$
then the mode is absolutely $\delta$-causal.
\label{thm:homogen}
\end{theorem}
\begin{proof}
If $\xi^* \in \Xi^*(W)$ and $\xi^*_j > 0$, the condition \Assume{5} and 
Lemma \ref{lemma:homogen} imply that $V - W_j > \delta \geq 0$.
\end{proof}

\begin{remark}
The case most frequently encountered in applications of SSPs is the
homogeneity of degree one.  When $d=1$, equation (\ref{eq:KT_3})
states that $\lambda = V$ and 
the condition \Assume{5} becomes even simpler
$$
\Assume{5'} \hspace*{5cm}
\frac{\partial C}{\partial \xi_j}(\xi)  > \delta \geq 0
\qquad \text{for } \forall \xi \in \Xi_n, \forall j \in I(\xi).
$$
Lemma \ref{lemma:homogen} and Theorem \ref{thm:homogen}
can be viewed as generalizations of the key idea in proofs of causality
in \cite{Tsitsiklis} and \cite[Appendix]{SethVlad3}.
\label{rem:homogen_one_common}
\end{remark}

\begin{remark}
If \Assume{4} holds and $C$ is strictly convex, then \Assume{5}
is a necessary condition for the absolute $\delta$-causality of the mode.
Indeed, suppose \Assume{5} is violated for some $\bar{\xi} \in \Xi_n, \, j \in I(\bar{\xi})$
and let $K = 1 + \max\limits_i \frac{\partial C}{\partial \xi_i}(\bar{\xi})$.
If for each $i=1, \ldots ,n$ we choose
$W_i = K - \frac{\partial C}{\partial \xi_i}(\bar{\xi})$,
this ensures that $W \in R_+^n$, $K = \lambda$, and 
$\Xi^*(W) = \{ \bar{\xi} \}$, which implies $V \leq W_j + \delta$
even though $\bar{\xi}_j > 0$.
\label{rem:sufficiency}
\end{remark}

\begin{lemma} Suppose the cost 
$$
\Assume{6} \hspace*{3cm}
C: \R_+^n \mapsto \R_+ \text{ is 
twice continuously differentiable.}
$$
Then for every $W \in \R_{+,0}^n, \xi^* \in \Xi^*(W), j \in I(\xi^*)$ 
there exists a point $\hat{\xi}$ 
on the straight line segment $[\e_j, \xi^*]$ such that 
$$ 
V - W_j \; = \; C(\e_j) \, - \, 
\frac{1}{2} (\e_j - \xi^*)^T H(\hat{\xi}) (\e_j - \xi^*),
$$
where $H$ is the Hessian matrix of $C(\xi)$.
\label{lemma:non_homogen}
\end{lemma}
\begin{proof}
If $\xi^* \in \Xi^*(W)$ and $j \in I(\xi^*)$, then 
the Kuhn-Tucker optimality conditions yield two different formulas 
(\ref{eq:KT_1}) and (\ref{eq:KT_2})
for the Lagrange coefficient $\lambda$. Combining these we see that
$$
V - W_j  = (V - \lambda) + (\lambda - W_j) =
\left( C(\xi^*) -  \sum\limits_{i=1}^n \xi^*_i 
 \frac{\partial C}{\partial \xi_i}(\xi^*)
\right) + \frac{\partial C}{\partial \xi_j}(\xi^*) =
C(\xi^*) + (\e_j - \xi^*)^T \nabla C(\xi^*).
$$
By Taylor's theorem there exists a point 
$\hat{\xi} \in [\e_j, \xi^*] \subset \Xi_n$ such that
$$
C(\e_j) = C(\xi^*) + (\e_j - \xi^*)^T \nabla C(\xi^*) + 
\frac{1}{2} (\e_j - \xi^*)^T H(\hat{\xi}) (\e_j - \xi^*);
$$
thus, 
$V - W_j = C(\e_j) - \frac{1}{2} (\e_j - \xi^*)^T H(\hat{\xi}) (\e_j - \xi^*).$
\end{proof}

\begin{theorem} 
Consider an $n$ by $(n-1)$ matrix $B$, whose columns form an orthonormal basis for 
the subspace orthogonal to $[1, \ldots,1]^T \, \in \R^n$.
Suppose the cost $C$ satisfies \Assume{6} and $\hat{H}(\xi) = B^T H(\xi) B$ is 
its projected Hessian.
If $\Lambda(\hat{H}(\xi))$ is the maximum eigenvalue of $\hat{H}(\xi)$ and 
$$
\Assume{7} \hspace*{3cm} 
\min\limits_i \; C(\e_i) \quad > \quad 
\delta \; + \; 
\max \left\{ 0, \;
\max\limits_{\xi \in \Xi_n} \; \Lambda \left(\hat{H}(\xi) \right)
\right\}
$$
then the mode is absolutely $\delta$-causal.
\label{thm:second_deriv_bound}
\end{theorem}
\begin{proof}
First, we assume that  
$\max\limits_{\xi \in \Xi_n} \, \Lambda \left(\hat{H}(\xi) \right) > 0$
(the other case is already covered by Theorem \ref{thm:concave}).
If $\xi^* \in \Xi^*(W)$ and $j \in I(\xi^*)$, then
there exists $\beta \in \R^{n-1}$ such that $(\e_j - \xi^*) = B \beta$.
We note that $\|\beta\| = \|\e_j - \xi^*\| \leq \sqrt{2}.$
Since the Lemma \ref{lemma:non_homogen}
applies, 
$$
V - W_j  =  C(\e_j)  - \frac{1}{2} \beta^T \hat{H}(\hat{\xi}) \beta
\geq 
C(\e_j) - \frac{1}{2} \|\beta\|^2 \Lambda \left(\hat{H}(\hat{\xi}) \right)
\geq
\min\limits_i \, C(\e_i) \, - \, 
\max\limits_{\xi \in \Xi_n} \, \Lambda \left(\hat{H}(\xi) \right)
> \delta.
$$
\end{proof}

\begin{remark}
Since the cost function is always evaluated on $\Xi_n$,
the condition \Assume{4} is somewhat awkward:
the cost can {\em always} be considered absolute homogeneous of degree one
since $C(\xi)$ can be replaced by 
$\tilde{C}(\xi) = \| \xi \|_{\scriptscriptstyle1} 
C \left( \frac{\xi}{\| \xi \|_{\scriptscriptstyle1}} \right)$, which has 
the same values as $C$ on $\Xi_n$.  A more meaningful question is:
supposing that $C$ is smooth and homogeneous of degree one, 
what additional conditions on $C$ and its directional derivatives 
inside $\Xi_n$ are sufficient to guarantee \Assume{5'} ?
It is easy to see that \Assume{7} is an answer to that question
since $\max\limits_{\xi \in \Xi_n} \, \Lambda \left(\hat{H}(\xi) \right)$ is the 
upper bound on the second derivative of $C$ restricted to any straight line in $\Xi_n$.
\label{rem:homogen_extend}
\end{remark}

\section{MSSPs approximating continuous deterministic problems.}
\label{s:HJB_discr}
As already mentioned, MSSPs naturally arise in approximations of
deterministic continuous optimal control problems.
To illustrate this, we consider a class of time-optimal trajectory 
problems.  
Many variants of these problems are studied in robotic navigation,
optimal control, and front propagation literature;
a detailed discussion of the version presented here can be found 
in \cite{SethVlad3}.

Suppose $\y(t) \in \R^2$ is the vehicle's position at the time $t$
and the vehicle starts at $\y(0) = \x$ inside the domain $\domain$.
We are free to choose any direction of motion (any vector in 
$S_1 = \left\{ \ba \in \R^2 \, \mid \; \| \ba \| = 1 \, \right\}$), but the 
speed will dependent on the chosen direction and on the current position
of the vehicle.  The vehicle's dynamics is governed by 
$\y'(t) = f(\y(t), \ba(t)) \ba(t)$, where $f: \R^2 \times S_1 \mapsto \R$
is a Lipschitz-continuous speed function satisfying
$0 < F_1 \leq f(\x, \ba)  \leq F_2$ for all $\x$ and $\ba$.
Additional exit-time-penalty $q$ is incurred at the boundary;
we will assume that $q: \boundary \mapsto \R$ is non-negative and 
Lipschitz-continuous.\\
The goal is to cross the boundary $\boundary$ as quickly as possible.

The value function of this problem is $u(\x)$ (the minimal-time-to-exit
after starting from $\x$).  It is well-known that $u(\x)$ is 
the unique viscosity solution \cite{CranLion} of the following
static Hamilton-Jacobi-Bellman PDE
\begin{equation}
\begin{array}{ll}
\max\limits_{\ba \in S_{1}} \{ (\nabla u(\x) \cdot (-\ba)) f(\x, \ba) \}
= 1,&
\hspace{1mm} \x \in \domain \subset \R^2\\
u(\x) = q(\x),&
\hspace{1mm} \x \in \boundary.\\
\end{array}
\label{eq:HJB}
\end{equation}
The optimal trajectories coincide with the characteristic curves of this PDE.
If the problem is {\em isotropic} (i.e., if $f(\x, \ba) = f(\x))$,
the above PDE is equivalent to the usual Eikonal equation 
$\| \nabla u(\x) \| f(\x) = 1$ and the optimal trajectories coincide with 
the gradient lines of $u(\x)$.

For simplicity we will first assume that the domain $\domain$ is rectangular 
and that $X$ is a uniform Cartesian grid on $\cdomain$.  
Concentrating on one particular gridpoint $\x \in X \cap \domain$,
we will number all of its neighbors as in Figure \ref{fig:Two_Stencils}.
Suppose that the optimal initial direction of motion $\ba$
lies in the first quadrant and assume that the corresponding optimal trajectory
remains a straight line until intersecting the segment $\x_1 \x_3$ at some
point $\xtilde$ (see Figure \ref{fig:Two_Stencils}A).  Then it follows that 
$$
u(\x) \; = \; \frac{\|\xtilde - \x\|}{f(\x, \ba)} \; + \; u(\xtilde).
$$
Let $\xtilde = \xi_1 \x_1 + \xi_2 \x_3$;  a linear approximation yields
$$
u(\x) \; \approx \; \frac{\|\xtilde - \x\|}{f(\x, \ba)} \; + \; \xi_1 u(\x_1) + \xi_2 u(\x_3).
$$
Of course, since $\xtilde$ is not a priori known, we would have to minimize over all 
possible intersection points and all four quadrants.

\begin{figure}[hhhh]
\center{
$
\begin{array}{cc}
\begin{tikzpicture}[scale=1.6]
\tikzstyle{gridpoint}=[circle,draw=black!100,fill=black!100,thick,inner sep=0pt,minimum size=1mm]
\tikzstyle{interpoint}=[circle,draw=black!100,fill=black!100,thick,inner sep=0pt,minimum size=0.7mm]
\draw[step=1cm,gray,very thin] (-1.4,-1.4) grid (1.4,1.4);

  \node[gridpoint]         (x)		at (0,0)	[label=120:{$\x$}]					{};
  \node[gridpoint]         (x1)		at (1,0)	[label=right:{$\x_1$}]				{};
  \node[gridpoint]         (x2)		at (1,1)	[label=above right:{$\x_2$}]		{};
  \node[gridpoint]         (x3)		at (0,1)	[label=above:{$\x_3$}]				{};
  \node[gridpoint]         (x4)		at (-1,1)	[label=above left:{$\x_4$}]			{};
  \node[gridpoint]         (x5)		at (-1,0)	[label=left:{$\x_5$}]				{};
  \node[gridpoint]         (x6)		at (-1,-1)	[label=below left:{$\x_6$}]			{};
  \node[gridpoint]         (x7)		at (0,-1)	[label=below:{$\x_7$}]				{};
  \node[gridpoint]         (x8)		at (1,-1)	[label=below right:{$\x_8$}]		{};

  \path (x)		edge	[thick]		node {}		(x1)
				edge	[thick]		node {}		(x3)
				edge	[thick]		node {}		(x5)
				edge	[thick]		node {}		(x7);
  \path (x1)	edge	[thick]		node {}		(x3)
				edge	[thick]		node {}		(x7);
  \path (x5)	edge	[thick]		node {}		(x3)
				edge	[thick]		node {}		(x7);

  \draw[->]		(0,0)	--	(2,1)	node[below] {$\ba$}; 

  \node[interpoint]       (xtilde)	at (0.6666,0.3333)	[label=above:{$\xtilde$}]	{};
\end{tikzpicture}&

\begin{tikzpicture}[scale=1.6]
\tikzstyle{gridpoint}=[circle,draw=black!100,fill=black!100,thick,inner sep=0pt,minimum size=1mm]
\tikzstyle{interpoint}=[circle,draw=black!100,fill=black!100,thick,inner sep=0pt,minimum size=0.7mm]
\draw[step=1cm,gray,very thin] (-1.4,-1.4) grid (1.4,1.4);

  \node[gridpoint]         (x)		at (0,0)	{};
  \node	at (-0.1,0.2) {$\x$};
  			
  \node[gridpoint]         (x1)		at (1,0)	[label=right:{$\x_1$}]				{};
  \node[gridpoint]         (x2)		at (1,1)	[label=above right:{$\x_2$}]		{};
  \node[gridpoint]         (x3)		at (0,1)	[label=above:{$\x_3$}]				{};
  \node[gridpoint]         (x4)		at (-1,1)	[label=above left:{$\x_4$}]			{};
  \node[gridpoint]         (x5)		at (-1,0)	[label=left:{$\x_5$}]				{};
  \node[gridpoint]         (x6)		at (-1,-1)	[label=below left:{$\x_6$}]			{};
  \node[gridpoint]         (x7)		at (0,-1)	[label=below:{$\x_7$}]				{};
  \node[gridpoint]         (x8)		at (1,-1)	[label=below right:{$\x_8$}]		{};

  \path (x)		edge	[thick]		node {}		(x1)
				edge	[thick]		node {}		(x2)
				edge	[thick]		node {}		(x3)
				edge	[thick]		node {}		(x4)
				edge	[thick]		node {}		(x5)
				edge	[thick]		node {}		(x6)
				edge	[thick]		node {}		(x7)
				edge	[thick]		node {}		(x8);
  \path (x1)	edge	[thick]		node {}		(x2)
				edge	[thick]		node {}		(x8);
  \path (x3)	edge	[thick]		node {}		(x4)
				edge	[thick]		node {}		(x2);
  \path (x5)	edge	[thick]		node {}		(x6)
				edge	[thick]		node {}		(x4);
  \path (x7)	edge	[thick]		node {}		(x8)
				edge	[thick]		node {}		(x6);

  \draw[->]		(0,0)	--	(2,1)	node[below] {$\ba$}; 

  \node[interpoint]       (xtilde)	at (1,0.5)		{};
  \node								at (1.1,0.4)	{$\xtilde$};	
\end{tikzpicture}
\\
A & B
\end{array}
$
}
\caption{
{\footnotesize
Two simple stencils using 4 nearest neighbors (A) or 8 nearest
neighbors (B) on a uniform Cartesian grid.
}
}
\label{fig:Two_Stencils}
\end{figure}
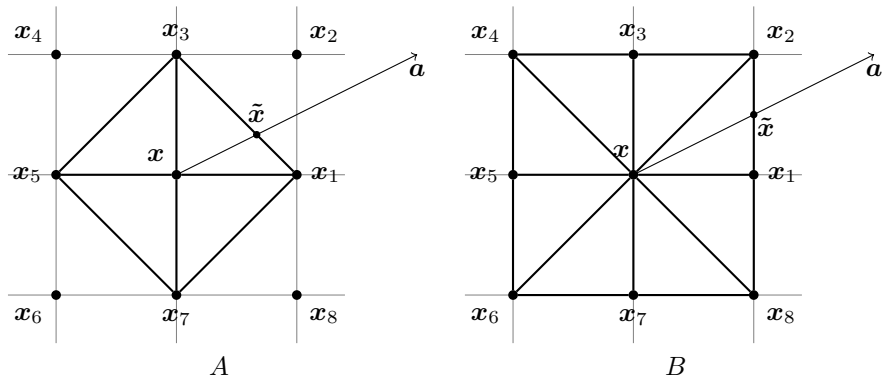

We will enumerate all quadrants as follows 
$\M(\x) = \{ 
(\x_1, \x_3), \, 
(\x_3, \x_5), \, 
(\x_5, \x_7), \, 
(\x_7, \x_1), \}$. 
For any $(\z^m_1, \z^m_2) \in \M$ and any $\xi \in \Xi_2$
we can similarly denote
$$
\xtilde_{\xi} = \xi_1 \z^m_1 + \xi_2 \z^m_2;
\qquad
\tau(\xi) = \|\xtilde_{\xi} - \x\|;
\qquad
\ba_{\xi} = (\xtilde_{\xi} - \x) / \tau(\xi).
$$
We can now state a semi-Lagrangian discretization of the PDE (\ref{eq:HJB}):
\begin{eqnarray}
\label{eq:Markov_aniso}
U(\x) & = &
\min\limits_{(\z^m_1, \z^m_2) \in \M(\x)} \,
\min\limits_{\xi \in \Xi_2}
 \left\{ 
\frac{\tau(\xi)}{f(\x, \ba_{\xi})}
\, + \, \left( \xi_1 U(\z^m_1) + \xi_2 U(\z^m_2) \right)
\right\},  \quad \text{ for } \forall \x \in X \cap \domain;\\
\nonumber
U(\x) &=& q(\x),
  \quad \text{ for } \forall \x \in X \cap \boundary.
\end{eqnarray}
This (fully deterministic) derivation is similar to the one used by Gonzales
and Rofman in \cite{GonzalezRofman}.  

On the other hand, it is easy to see
that this system of equations also describes the value function for 
an MSSP on $X \cup \{\T\}$:
\begin{itemize} 
\item
for the nodes $\x \in X \cap \boundary$, there is a single (deterministic)
transition to $\T$ with the cost $q(\x)$;
\item
for the nodes $\x \in X \cap \domain$,
the set of quadrants $\M(\x)$ can be interpreted as a set of modes 
and
$C^m(\x, \xi) = \tau(\xi) \, / \, f(\x, \ba_{\xi}).$
\end{itemize}
This interpretation is in the spirit of Kushner's and Dupuis' approach of
approximating continuous optimal control by controlled Markov processes
\cite{KushnerDupuis}.

On a uniform Cartesian grid and the stencil of Figure \ref{fig:Two_Stencils}A, 
we can express $\tau(\xi) = h \sqrt{\xi_1^2 + \xi_2^2}$,
where $h$ is the grid size.  If the problem is isotropic, the cost function becomes
$C(\x, \xi) = (h / f(\x)) \sqrt{\xi_1^2 + \xi_2^2}.$  
A similar construction in $\R^n$ leads to modes containing $n$ neighbor-nodes each
and the cost function
\begin{equation}
C(\x, \xi) = \frac{h}{f(\x)} \sqrt{\sum\limits_{i=1}^{n} \xi_i^2}.  
\label{eq:Tsi_cost1}
\end{equation}
This function is homogeneous of degree one in terms of $\xi$; moreover, 
$
\frac{\partial C}{\partial \xi_j}(\x, \xi) = \frac{h^2}{f(\x)} \frac{\xi_j}{\tau(\xi)},
$
which is positive if and only if $\xi_j > 0$.
By the Theorem \ref{thm:homogen}, each mode is absolutely causal
and a Dijkstra-like method can be used to solve the problem.
This is in fact the first of two methods introduced by Tsitsiklis in 
\cite{Tsitsiklis}.
Since this $C$ is also convex in $\xi$, 
the Remark \ref{rem:sufficiency} shows that the modes are not
absolutely $\delta$-causal for any $\delta > 0$; hence, the Dial's method 
is generally not applicable.

Another obvious computational stencil in $\R^2$ uses all 8 neighboring gridpoints
as shown in Figure \ref{fig:Two_Stencils}B.  Here the optimal trajectory is still assumed to
remain a straight line until the intersection with a segment,
but the list of segments is different:
$$
\M(\x) = \left\{ 
(\x_1, \x_2), \, 
(\x_3, \x_2), \, 
(\x_3, \x_4), \, 
(\x_5, \x_4), \, 
(\x_5, \x_6), \, 
(\x_7, \x_6), \, 
(\x_7, \x_8), \, 
(\x_1, \x_8) 
\right\}.
$$
The discretized equation (\ref{eq:Markov_aniso}) still holds,
but the difference is that
$$
\tau(\xi) = \|\xtilde_{\xi} - \x\| = 
\|(\xi_1 \z^m_1 + \xi_2 \z^m_2) - \x\| = h \sqrt{1 + \xi_2^2}.
$$
If the problem is isotropic, the cost function becomes
$C(\x, \xi) = (h / f(\x)) \sqrt{1 + \xi_2^2}.$ 
Theorem \ref{thm:second_deriv_bound} is certainly applicable, 
but instead we re-write the above as a function homogeneous 
of degree one (see Remark \ref{rem:homogen_extend}):
$C(\x, \xi) = (h / f(\x)) \sqrt{(\xi_1 + \xi_2)^2 + \xi_2^2}.$ 
We now note that
$$
\frac{\partial C}{\partial \xi_1}(\x, \xi) = 
\frac{h^2}{f(\x)} \frac{1}{\tau(\xi)};
\qquad
\frac{\partial C}{\partial \xi_2}(\x, \xi) = 
\frac{h^2}{f(\x)} \frac{1 + \xi_2}{\tau(\xi)} 
\geq \frac{\partial C}{\partial \xi_1}(\x, \xi).
$$
Since $\tau(\xi) \leq h \sqrt{2}$, we conclude that
\begin{equation}
\frac{\partial C^m}{\partial \xi_j}(\x, \xi) \geq 
\frac{h}{F_2 \sqrt{2}} = \delta  > 0,
\qquad
\text{for }
\forall \x \in X \cap \domain,
\forall m \in \M(\x),
\forall \xi \in \Xi_2, \,
j=1,2.
\label{eq:Tsi2_delta}
\end{equation}
By the Theorem \ref{thm:homogen}, each mode is absolutely $\delta$-causal
and a Dial-like method can be used with buckets of width $\delta$ 
(corresponding to the second method introduced in \cite{Tsitsiklis}).

More generally,  suppose that $X$ is a simplicial mesh on $\cdomain \subset \R^n$
with the minimum edge-length of $h$.  Let $S(\x)$ be the set of all simplexes
in the mesh that use $\x$ as one of the vertices.  Each such simplex $s \in S(\x)$
corresponds to a single mode $m \in \M(\x)$ consisting of all other  vertices of $s$
besides $\x$.
For any mode $m = (\z^m_1, \ldots, \z^m_n)$ and any $\xi \in \Xi_n$ we can similarly define
$$
\xtilde^m_{\xi} = \sum\limits_{i=1}^{n} \xi_i \z^m_i;
\qquad
\tau^m(\xi) = \|\xtilde^m_{\xi} - \x\|;
\qquad
\ba^m_{\xi} = (\xtilde^m_{\xi} - \x) / \tau^m(\xi).
$$

Since 
$
\tau^m(\xi) = 
\sqrt{\sum\limits_{i=1}^{n} \sum\limits_{k=1}^{n} \xi_i \xi_k (\z^m_i - \x)^T (\z^m_k - \x)},
$ 
we see that 
$$
\frac{\partial \tau^m}{\partial \xi_j}(\xi) =
\frac{
(\z^m_j - \x)^T
\left(
\sum\limits_{i=1}^{n} \xi_i (\z^m_i - \x)
\right)
}
{\tau^m(\xi)} 
= 
\frac{
(\z^m_j - \x)^T
(\xtilde^m_{\xi} - \x)
}
{\tau^m(\xi)} 
=
(\z^m_j - \x)^T \ba^m_{\xi}
= \|\z^m_j - \x\| \cos \beta_{\xi, j},
$$
where $\beta_{\xi, j}$ is the angle between
$\ba^m_{\xi}$ and $(\z^m_j - \x)$.
Suppose $\beta(\x, m)$ is the maximum angle between a pair of vectors 
$(\z^m_k - \x)$ and $(\z^m_i - \x)$ maximizing over all 
$i,k \in \{1, \ldots, n \}.$  Furthermore, define
$$
\beta(\x) = \max\limits_{m \in \M(\x)} \beta(\x, m);
\qquad
\beta = \max\limits_{\x \in X \cap \domain} \beta(\x).
$$
Since $\ba^m_{\xi}$ lies in the cone defined by 
$(\z^m_1 - \x), \ldots, (\z^m_n - \x)$, 
we know that $\beta_{\xi, j} \leq \beta(\x, m) \leq \beta(\x) \leq \beta.$
Therefore,
$$
(\beta < \frac{\pi}{2})
\qquad
\Longrightarrow
\qquad
\frac{\partial \tau^m}{\partial \xi_j}(\xi) = 
\|\z^m_j - \x\| \|\ba^m_{\xi}\| \cos \beta_{\xi, j} \geq h \cos \beta > 0.
$$
The dynamic programming equations in this case become
\begin{eqnarray}
\label{eq:Markov_aniso_mesh}
U(\x) & = &
\min\limits_{m \in \M(\x)} \,
\min\limits_{\xi \in \Xi_n}
 \left\{ 
\frac{\tau^m(\xi)}{f(\x, \ba^m_{\xi})}
\, + \, \left( \sum\limits_{i=1}^{n} \xi_i U(\z^m_i) \right)
\right\},  \quad \text{ for } \forall \x \in X \cap \domain;\\
\nonumber
U(\x) &=& q(\x),
  \quad \text{ for } \forall \x \in X \cap \boundary.
\end{eqnarray}
The cost function 
$C^m(\x, \xi) = \tau^m(\xi) \, / \, f(\x, \ba^m_{\xi})$
is homogeneous of degree one in $\xi$.
For the isotropic case,
we see that
$$
\frac{\partial C^m}{\partial \xi_j}(\x, \xi) = 
\frac{1}{f(\x)}\frac{\partial \tau^m}{\partial \xi_j}(\xi)
= \frac{\|\z^m_j - \x\| \cos \beta_{\xi, j}}{f(\x)}.
$$
Thus, for the Eikonal PDE on any acute mesh (i.e., for $\beta < \frac{\pi}{2}$), 
each mode of the discretization is absolutely causal 
by Theorem \ref{thm:homogen} and a Dijkstra-like method is applicable
(this is a re-derivation of the result in \cite[Appendix]{SethVlad3}).
Moreover, if $\beta < \frac{\pi}{2}$ then 
$
\frac{\partial C^m}{\partial \xi_j}(\x, \xi) \geq
\frac{h \cos \beta}{F_2} = \delta > 0.
$
This provides the optimal bucket-width $\delta$ to use in a Dial-like
method when solving the Eikonal PDE on any acute mesh.  As far as we know,
no general formula for $\delta$ has been derived elsewhere up till now.

We note that the last result is applicable even in a more general 
situation, when the computational stencil $S(\x)$ does not correspond
to a set of non-overlapping simplexes present in the mesh.  E.g.,
for the example in Figure \ref{fig:Two_Stencils}B, 
$\, \beta = \pi/4$ and this yields the same $\delta$ as 
in (\ref{eq:Tsi2_delta}).
That leads to an interesting dilemma:  including more nearby nodes
into a computational stencil usually decreases $\beta$ and increases
the bucket-width thus reducing the total number of ``bucket-acceptance'' 
steps
until the termination of Dial's algorithm.  On the other hand, a larger
$S(\x)$ increases both the computational complexity of a single step
(more tentative labels to update after each acceptance) and
the discretization error (proportional to $h$ in the above examples).
Finding an optimal way for handling this trade-off, could
further speed-up non-iterative methods for Eikonal PDEs on acute meshes.
We note that $h/F_2$ remains the upper bound for $\delta$ and corresponds to
the situation when the vehicle is allowed to move only along the directions
$(\z^m_j - \x)$.

A much harder question is the applicability of label-setting methods
to semi-Lagrangian discretizations of anisotropic optimal control problems.
It is well-known that equations (\ref{eq:Markov_aniso_mesh}) are 
generally not causal; this issue is discussed in detail in 
\cite{SethVlad3, Vlad_Thesis}.
On uniform Cartesian grids, the criteria for applicability of a Dijkstra-like 
method to anisotropic problems were previously provided in \cite{SethBook2},
\cite{OsherFedkiw_tutorial}, and more recently in \cite{AltonMitchell_TR}.  
All of these criteria are grid-orientation dependent;
i.e., given a Hamilton-Jacobi PDE, its semi-Lagrangian or Eulerian 
discretization may or may not be computed correctly by a Dijkstra-like
method depending on whether the anisotropy in the PDE is aligned 
with the grid directions.
Here we provide a criterion for applicability of a Dijkstra-like method for
discretizations 
based on arbitrary acute stencils.
In the anisotropic case,
$$
\frac{\partial C^m}{\partial \xi_j}(\x, \xi) = 
\frac{
f(\x, \ba_\xi) \frac{\partial \tau^m}{\partial \xi_j}(\xi)
\, - \,
\tau^m(\xi) \frac{\partial f}{\partial \xi_j}(\x, \ba_\xi)}
{f^2(\x, \ba_\xi)}.
$$
Suppose that there exists $\delta \geq 0$
such that for 
$
\forall \x \in X \cap \domain, \,
\forall m \in \M(\x), \,
\forall \xi \in \Xi_n, \,
\forall j \in \{ 1,\ldots, n \}
$
$$
(\xi_j > 0)
\qquad \Longrightarrow \qquad
\frac{\partial f}{\partial \xi_j}(\x, \ba_\xi) < 
\frac{
f(\x, \ba_\xi) 
\left[
\frac{\partial \tau^m}{\partial \xi_j}(\xi)
\, - \, \delta f(\x, \ba_\xi)
\right]
}
{\tau^m(\xi)}.
$$
By the Theorem \ref{thm:homogen}, this implies that
a Dijkstra-like method will be applicable and 
a Dial-like method will also be applicable if $\delta >0$.
Building label-setting methods based on this sufficient condition
could potentially yield algorithms outperforming 
the Ordered Upwind Methods specially designed to restore the
causality of anisotropic problems by dynamically extending 
the stencil \cite{SethVlad2, Vlad_Thesis, SethVlad3}.
We intend to explore this approach in the future work.

\section{Conclusions.}
\label{s:conclusions}
We defined a large class of Multimode Stochastic Shortest Path problems
and derived a number of sufficient conditions to check the applicability
of the label-setting methods. We illustrated the usefulness of our approach
to the numerical analysis of first-order non-linear boundary value problems by
reinterpreting previous label-setting methods for the Eikonal PDE
on Cartesian grids.  For Eikonal equation on arbitrary meshes, we re-interpreted 
the prior Dijkstra-like method and derived the new formula of bucket-width
for Dial-like methods.
We also developed a new sufficient condition for the applicability of 
label-setting methods to anisotropic Hamilton-Jacobi PDEs on arbitrary stencils.

In practice, the applicability of label-setting methods to a particular SSP
can be tested directly in $O(M)$ operations: 
upon the method's termination, a single value iteration can be applied
and, if it results in no changes, the value function was computed correctly.
However, the sufficient conditions (presented above for MSSPs) allow
to avoid these additional computations.

Unfortunately, the framework of MSSPs is not flexible enough to express many common 
discrete stochastic control problems, where not all possible
probability distributions over successor nodes are available.
Nevertheless, we hope that the key idea of our approach (splitting the original
MSSP into a number of absolutely causal auxiliary problems) can be generalized
to test the applicability of label-setting methods to other SSPs.
Since SSPs can be naturally extended to describe stochastic games on graphs 
\cite{Patek_Berts}, we also intend to investigate the applicability of our 
approach to the latter.  
If successful, this will potentially yield 
efficient numerical methods for a wide class of first and second order static 
Hamilton-Jacobi equations.  

In Dial-like methods, the bucket width can be sometimes adjusted on the
fly based on the not-yet-accepted part of the problem only.  
We expect such extensions to be advantageous for problems, where 
the cost function $C$ has very different lower bounds for different nodes.
Another open question of practical 
importance is the use of label-setting methods to obtain an {\em approximation}
of the value function for non-causal SSPs.  
Recently, a numerical method based on a related idea was introduced 
in \cite{Sapiro_buckets} for Eikonal PDEs :
a Dial-like method is used with buckets of width $\delta$ for a discretization
that is not $\delta$-causal.  This introduces additional errors (analyzed
in \cite{RaschSatzger}), but decreases the method's running time.

Finally, the performance comparison of label-setting and label-correcting 
methods on MSSPs is a yet another interesting topic for the future research.

\acknowledgments{The author would like to thank B. Van Roy, M.J. Todd, and D. Shmoys.
The author's research is supported by 
the National Science Foundation grants DMS-0514487 and CTS-0426787.}


\small
\begin{thebibliography}{99}

\bibitem{Ahuja}
Ahuja, R.K., Magnanti, T.L., \& Orlin, J.B.,
{\it Network Flows},
Prentice Hall, Upper Saddle River, NJ, 1993.


\bibitem{AltonMitchell_TR}
K. Alton \& I. M. Mitchell,
{\it Fast marching methods for a class of anisotropic stationary
Hamilton-Jacobi equations},
Technical Report TR-2006-27, Department of Computer Science,
University of British Columbia, 2007.





\bibitem{Bertsekas_NObook} Bertsekas, D.P.,
{\it Network  optimization: continuous \& discrete models},
Athena Scientific, Boston, MA, 1998. 

\bibitem{Bertsekas_DPbook} Bertsekas, D.P.,
{\it Dynamic Programming and Optimal Control},
2nd Edition,  Volumes I and II,
Athena Scientific, Boston, MA, 2001. 

\bibitem{BertsTsi} Bertsekas, D.P. \& Tsitsiklis, J.N.,
{\it An analysis of stochastic shortest path problems}, 
Mathematics of Operations Research, 16(3), pp.580-595, 1991. 

\bibitem{BerTsi_NDP} Bertsekas, D.P. \& Tsitsiklis, J.N.,
{\it Neuro-Dynamic Programming},
Athena Scientific, Boston, MA, 1996. 

\bibitem{Bonet} B. Bonet,
{\it On the Speed of Convergence of Value Iteration on 
Stochastic Shortest-Path Problems},
Mathematics of Operations Research, 32(2), pp.365-373, 2007.

\bibitem{BorRasch}
F. Bornemann and C. Rasch,
{\it Finite-element Discretization of Static Hamilton-Jacobi
Equations based on a Local Variational Principle},
Computing and Visualization in Science, 9(2), pp.57-69, 2006.


\bibitem{Branicky1} Branicky, M.S. \& Hebbar, R.,
{\it Fast marching for hybrid control},
Proceedings, IEEE Intl. Symp. Computer Aided Control System Design,
Kohala-Kona, HI, August 22-27, 1999.









\bibitem{CranLion} Crandall, M.G. \& Lions, P-L.,
{\it Viscosity Solutions of Hamilton-Jacobi Equations},
Tran. AMS, 277, pp. 1-43, 1983.



\bibitem{Dial} R. Dial,
{\it Algorithm 360: Shortest path forest with topological ordering},
Comm. ACM, pp. 632--633, 1969.

\bibitem{Diks} E.W. Dijkstra, 
{\it A Note on Two Problems in Connection with Graphs},
Numerische Mathematik, 1 (1959), pp. 269--271.







\bibitem{Falcone_InfHor} M. Falcone, 
{\it A Numerical Approach to the Infinite Horizon
Problem of Deterministic Control Theory},
Applied Math. Optim., 15 (1987), pp. 1--13;
corrigenda 23 (1991), pp. 213--214.




\bibitem{GonzalezRofman} Gonzales, R. \& Rofman, E.,
{\it On Deterministic Control Problems: an Approximate Procedure
for the Optimal Cost, I, the Stationary Problem},
SIAM J. Control Optim., 23, 2, pp. 242-266, 1985.









\bibitem{KimGMM} Kim, S.,  
{\it An O(N) level set method for eikonal equations},
SIAM J. Sci. Comput., 22, pp. 2178-2193, 2001.

\bibitem{KimmSethTria} Kimmel, R. \& Sethian, J.A.,
{\it Fast Marching Methods on Triangulated Domains},
Proc. Nat. Acad. Sci., 95, pp. 8341-8435, 1998.





\bibitem{Kushner_games} H.J. Kushner,
{\it Numerical Approximations for Stochastic Differential Games},
SIAM J. Control Optim., 41(2), pp. 457-486, 2002.

\bibitem{KushnerDupuis} H.J. Kushner \& P.G. Dupuis,
{\it Numerical Methods for Stochastic Control Problems
in Continuous Time}, Academic Press, New York, 1992.









\bibitem{OsherFedkiw_tutorial} Osher, S. \& Fedkiw, R.P., 
{\it Level Set Methods: An Overview and Some Recent Results}, 
IPAM GBM Tutorials, March 27 - April 6, 2001.


\bibitem{Patek_Berts}
S. D. Patek and D. P. Bertsekas, 
{\it Stochastic Shortest Path Games},
SIAM J. Control Optim., 37(3), pp.804--824, 1999.

\bibitem{PolyBerTsi}
L. C. Polymenakos, D. P. Bertsekas, and J. N. Tsitsiklis, 
{\it Implementation of Efficient Algorithms for Globally Optimal Trajectories},  
IEEE Transactions on Automatic Control, 43(2), pp. 278--283, 1998.





\bibitem{RaschSatzger}
C. Rasch and T. Satzger,
{\it Remarks on the $O(N)$ Implementation of the Fast Marching Method},
preprint, 2007.
\url{http://www.citebase.org/abstract?id=oai:arXiv.org:cs/0703082}



\bibitem{SethFastMarcLeveSet} J.A. Sethian,
{\it A Fast Marching Level Set Method for Monotonically Advancing Fronts},
Proc. Nat. Acad. Sci., 93, 4, pp. 1591--1595, February 1996.

\bibitem{SethBook2} J.A. Sethian,
{\it Level Set Methods and Fast Marching Methods: Evolving Interfaces in
Computational Geometry, Fluid Mechanics, Computer Vision and Materials
Sciences}, Cambridge University Press, 1996.

\bibitem{SethSIAM} Sethian, J.A.,
{\it Fast Marching Methods},
SIAM Review, Vol. 41, No. 2, pp. 199-235, 1999.

\bibitem{SethVlad1} J.A. Sethian \& A. Vladimirsky,
{\it Fast Methods for the Eikonal and Related Hamilton--Jacobi Equations on
Unstructured Meshes},
Proc. Nat. Acad. Sci., 97, 11 (2000), pp. 5699--5703.

\bibitem{SethVlad2} J.A. Sethian \& A. Vladimirsky,
{\it Ordered Upwind Methods for Static Hamilton-Jacobi Equations},
Proc. Nat. Acad. Sci., 98, 20 (2001), pp. 11069--11074.

\bibitem{SethVlad3} J.A. Sethian \& A. Vladimirsky,
{\it Ordered Upwind Methods for Static Hamilton-Jacobi Equations: 
Theory \& Algorithms},
SIAM J. on Numerical Analysis 41, 1 (2003), pp. 325-363.

\bibitem{SethVladHybrid} Sethian, J.A. \& Vladimirsky, A.,
{\it Ordered Upwind Methods for Hybrid Control},
5th International Workshop, HSCC 2002, Stanford, CA, USA, 
March 25-27, 2002, Proceedings (LNCS 2289). 



\bibitem{Stoppard} T. Stoppard,
{\it Rosencrantz and Guildenstern Are Dead}, London, Faber, 1967.


\bibitem{SzpiroDupuis}
Szpiro, A. \& Dupuis, P.,
{\it Second Order Numerical Methods for First Order Hamilton--Jacobi Equations},
SIAM J. on Numerical Analysis, 40(3), pp. 1136-1183, 2002. 


\bibitem{Tsitsiklis_conference} J.N. Tsitsiklis, 
{\it Efficient algorithms for globally optimal trajectories},
Proceedings, IEEE 33rd Conference on Decision and Control, pp. 1368--1373,
Lake Buena Vista, Florida, December 1994.

\bibitem{Tsitsiklis} J.N. Tsitsiklis, 
{\it Efficient Algorithms for Globally Optimal Trajectories},
IEEE Tran. Automatic Control, 40 (1995), pp. 1528--1538.



\bibitem{Vlad_Thesis} A. Vladimirsky,
{\it Fast Methods for Static Hamilton-Jacobi Partial Differential 
Equations},
Ph.D. Dissertation, Dept. of Mathematics, Univ. of California, Berkeley, 
2001.




\bibitem{Sapiro_buckets}
L. Yatziv, A. Bartesaghi, \& G. Sapiro, 
{\it $O(N)$ implementation of the fast marching algorithm}, 
J. Comput. Phys. 212 (2006), no. 2, 393-399.


\end{thebibliography}
\end{document}